\documentclass[a4paper,reqno]{amsart}
\usepackage{latexsym}
\usepackage{amssymb,amsthm,amsmath}
\usepackage[dvips]{graphicx}
\usepackage{xcolor}

\usepackage{todonotes}

\usepackage{tcolorbox}   %% colored background
\definecolor{mycolor}{rgb}{0.122, 0.435, 0.698}% Rule colour
\newcommand{\mybox}[1]{%
  \begin{tcolorbox}[colframe=mycolor,boxrule=0.2pt,arc=4pt,left=0pt,right=0pt,top=0pt,bottom=0pt,boxsep=2pt] #1
  \end{tcolorbox}%
}
\calclayout

\theoremstyle{plain}
\newtheorem{theorem}{Theorem}
\newtheorem{lemma}[theorem]{Lemma}
\newtheorem{corollary}[theorem]{Corollary}

\theoremstyle{definition}
\newtheorem{definition}[theorem]{Definition}
\theoremstyle{remark}
\newtheorem{remark}[theorem]{Remark}

\DeclareMathOperator*{\esssup}{\rm ess\,sup}

\def\d#1{{#1\kern-0.4em\char"16\kern-0.1em}}
\def\D#1{{\raise0.2ex\hbox{-}\kern-0.4em #1}}
\reversemarginpar

\newcounter{zd}

\newcounter{zdr}[subsection]

\def\sgn{\operatorname{sgn}}

\newcommand{\eps}{\varepsilon}

\def\mff{{\mathfrak f}}

\def\Div{{\rm div}}

\def\pa{\partial}

\def\cal{\mathcal}
\let\mib=\boldsymbol

\def\R{\mathbb{R}}    %promjena po sugestiji recezenta
\def\N{\mathbb{N}}    %promjena po sugestiji recezenta
\def\C{\mathbb{C}}    %promjena po sugestiji recezenta

\def\eps{\varepsilon}

\def\malpha{{\mib \alpha}}

\def\mx{{\bf x}}
\def\mxi{{\mib \xi}}

\def\my{{\bf y}}

\def\oi#1#2{(#1,#2)}   %promjena po sugestiji recezenta

\def\Rd{{\mathbb{R}^{d}}}    %promjena po sugestiji recezenta
\def\Sdmj{{\rm S}^{d-1}}
\def\Sd{{\rm S}^{d}}
\def\Sdj{{\rm S}^{d-1}}

\def\supp{{\rm supp\,}}

\begin{document}
%\today

\title[Heterogeneous degenerate parabolic equations]{Degenerate parabolic equations --  compactness and regularity of solutions}

\author{M.~Erceg}\address{Marko Erceg,
Department of Mathematics, Faculty of Science, University of Zagreb, Bijeni\v{c}ka cesta 30,
10000 Zagreb, Croatia}\email{maerceg@math.hr}

\author{D.~Mitrovi\'{c}}\address{Darko Mitrovi\'{c},
University of Vienna, Faculty of Mathematics, Oscar Morgenstern--platz 1,
1090 Vienna, Austria and\\
University of Montenegro, Faculty of Mathematics, Cetinjski put bb,
81000 Podgorica, Montenegro}\email{darko.mitrovic@univie.ac.at}



\begin{abstract}
We introduce a new method which resolves the problem of regularity and compactness of entropy solutions for nonlinear degenerate parabolic equations under non--degeneracy conditions on the sphere. In particular, we address a problem of regularity of entropy solutions to homogeneous (autonomous) degenerate parabolic PDEs and existence of weak solutions to heterogeneous degenerate parabolic PDEs (non-autonomous PDEs -- with flux and diffusion explicitly depending on the space and time variables).

The method of proof is reduction of the equation to a specific kinetic formulation involving two transport equations, one of the second and one of the first order. In the heterogeneous situation, this enables us to use (variants of) the H-measures to get velocity averaging lemmas and then, consequently, existence of a weak solution. In the homogeneous case, the kinetic reformulation makes it possible to get necessary estimates of the solution on the Littlewood-Paley dyadic blocks of the dual space.    
\end{abstract}

\maketitle

\section{Introduction}\label{sec:intro}

We consider the degenerate parabolic equation of the standard form

\begin{equation}\label{d-p}
\begin{split}
\pa_t u + \Div_\mx \bigl(\mff(t,\mx, u(t,\mx)) \bigr) =& \,\Div_\mx  \left( a(t,\mx, u(t,\mx)) \nabla_\mx u \right) \qquad \hbox{in} \ {\cal D}'(\R_+^{d+1})\;,
\end{split}
\end{equation}where $(t,\mx)\in \R_+^{d+1}:=[0,\infty)\times \R^d$, $u\in L_{loc}^2(\R_+^{d+1})$ is an unknown function, $\mff \in L^p(\R^{d+1}_+;C^1(\R))^d$, $p>1$, and $a\in C^1(\R^{d+1}_+\times \R;\R^{d\times d})$ are the coefficients whose properties we specify in what follows.  

To this end, the term 
$$
\Div_\mx( \mff(t,\mx,u(t,\mx)))
$$describes the convection/advection type movements corresponding to the modeled physical process, while the term 
$$
\Div_\mx  \left( a(t,\mx, u(t,\mx)) \nabla_\mx u \right)
$$ represents the evolution of the unknown quantity due to the diffusion effects (random motion of the unknown quantity from the areas where $u$ has smaller values toward the parts where it is higher). Since $a$ depends on $t$ and $\mx$, this implies that we are considering the process in an inhomogeneous medium (i.e.~its properties change with the position), which also demonstrates so-called memory effects (as the properties change with the time). Therefore, we refer to this equation as heterogeneous. One can also find the notion non-autonomous (see e.g. \cite{NPS}).

To be more precise, we assume
\begin{itemize}

\item[a)] $\mff \in L^p_{loc}(\R^{d+1}_+; C^1(\R)^d)$ and $f:= \pa_\lambda \mff \in L^p_{loc}(\R^{d+1}_+; C(\R))$, for some $p>1$;

\item[b)] $a \in C^{0,1}(\R^{d+1}_+\times \R; \R^{d\times d})$ and there exists $\sigma\in L^\infty(\R^{d+1}_+\times \R; \R^{d\times d})$ such that $\sigma=\sigma^T$ and $a=\sigma^2$ (both pointwisely on the whole domain). 
Thus, for $a(t,\mx,\lambda)=(a_{ij}(t,\mx,\lambda))_{i,j=1,\dots,d}$, we have
\begin{align}
\label{square-root}
a_{ij}(t,\mx,\lambda)&=\sum\limits_{k=1}^d \sigma_{ik}(t,\mx,\lambda)\sigma_{kj}(t,\mx,\lambda) \;; \\
\label{non-negativity}
\langle a(t,\mx,\lambda) \mxi\,|\,\mxi\rangle&=\sum\limits_{k=1}^d \left(\sum\limits_{j=1}^d \sigma_{kj}(t,\mx,\lambda)\xi_j \right)^2 \geq 0, \ \ \mxi \in \R^d \,,
\end{align}
i.e.~$a$ is a Lispchitz continuous pointwise symmetric positive semi-definite matrix function.

\item[c)] For every $\Omega\subset\subset \R^{d+1}_+$ and $K\subset\subset \R$, the non-degeneracy condition holds (see \cite[Theorem C]{LPT}):
\begin{equation}
\label{lp-111}
\esssup\limits_{(t,\mx)\in \Omega} \sup\limits_{(\xi_0,\mxi) \in \mathrm{S}^{d}} 
	\operatorname{meas}\bigl\{\lambda \in K :\, \xi_0+\langle f(t,\mx,\lambda)\,|\,\mxi
	\rangle=\langle a(t,\mx,\lambda)\mxi \,|\,\mxi \rangle=0 \bigr\}=0 	\,,
\end{equation}
where $\Sd$ denotes the unit sphere at the origin in $\R^{d+1}$, 
i.e.~$\mathrm{S}^d=\{(\xi_0,\mxi)\in\R^{d+1}:|(\xi_0,\mxi)|=1\}$.

\end{itemize} 

The equation describes various physical processes and therefore it appears in a wide spectrum of applications among which we mention flow in porous media (e.g. \cite{19CK, JMN}), sedimentation--consolidation processes \cite{9CK}, finances \cite{Polidoro}, etc. However, in most of the previous works, the coefficients are independent of $(t,\mx)$ which significantly restricts the range of models. To this end, we note that even in ideal laboratory conditions, one cannot achieve completely homogeneous medium \cite{JMN-1}.

From purely mathematical point of view, equation \eqref{d-p} is a parabolic equation which strongly degenerates meaning that there exists a moment $t\geq 0$ such that at some space point $\mx$ there exists a direction $\mxi\in \R^d \backslash \{0\}$ and the state $\lambda\in \R$ such that
$$
\langle a(t,\mx,\lambda) \mxi\,|\, \mxi \rangle=\sum\limits_{k,j=1}^d a_{kj}(t,\mx,\lambda) \xi_j \xi_k =0 \,.
$$
This essentially locally reduces the equation to a scalar conservation law and therefore one cannot expect, as in the case of strictly parabolic equations, solutions to remain regular along entire temporal axis. This causes a problem of uniqueness of  solutions unless appropriate entropy conditions are imposed. Such conditions appeared in \cite{VH} and they are by now a standard accompanying relation when considering \eqref{d-p} (see \eqref{e-c-n} below). 

In \cite{VH}, the authors resolved well-posedness of the Cauchy problem to \eqref{d-p} under regularity assumptions on the coefficients and BV--assumptions of the initial data (which imply BV--property of the entropy solution). Several generalisation under the mere assumption of boundedness of initial data were achieved later under special conditions on the diffusion matrix $a$ (it is diagonal or in the form of Laplacian of a scalar function; see e.g. \cite{Car, CdB, 24CK}) or in the one-dimensional situation \cite{4CK, 37CK}. The most general case of \eqref{d-p}, but with the coefficients independent of $(t,\mx)$ was resolved in \cite{CP}. The extension to the fully heterogeneous setting (with regular coefficients) is later achieved in \cite{BK, CK}. The ``breakthrough'' novelty in \cite{CP, CK} is the chain rule (see \cite[ Definition 2.1 (ii)]{CP} and \cite[Definition 2.1 (D.2.)]{CK}) and it is also a substantial ingredient of the proofs here.

In the current contribution, the flux $\mff$ appearing in \eqref{d-p} is discontinuous which significantly complicates the situation. By the best of our knowledge, the uniqueness question for \eqref{d-p} in the case of discontinuous flux is essentially intact and there are only handful of existence results under the discontinuous flux assumptions (see \cite{EMM, HKMP, LM2, MM, pan_jms}). However, in all of them one has either $(t,\mx)$-independent diffusion \cite{EMM, HKMP} or the diffusion with fixed degeneracy directions (so-called ultra-parabolic equations \cite{LM2, pan_jms, pan_aihp}). To be more succinct, let us now (informally) formulate the first main result of the paper. The result is precisely formulated and proved in Section 4. We note again that we are working with degenerate parabolic equations where neither the diffusion matrix is homogeneous nor the degeneracy direction is fixed. 

\mybox{
{\bf Theorem A.} If the coefficients $\mff$ and $a$ satisfy conditions {\rm a), b)}, and {\rm c)}, then for any $u_0 \in L^\infty(\Rd)\cap L^1(\R^d)$ equation \eqref{d-p} augmented with $u|_{t=0}=u_0(\mx)$ admits a weak solution.
}  

The main method when dealing with \eqref{d-p} with discontinuous flux is reduction of the equation to its kinetic counter-part (see \eqref{kinetic-n} here) and then proving a velocity averaging result. The velocity averaging phenomenon means that by averaging out a function  with respect to one of the variables we obtain a function of increased regularity. A little bit less general situation is to consider a  sequence of functions and then to obtain certain compactness properties of averages of the sequence of functions. This is initially noticed in \cite{Ago, nGPS} for transport equations of the type
\begin{equation}
\label{transport}
\begin{split}
\pa_t h_n + \Div_\mx \bigl(f(\lambda) h_n \bigr) =\pa_{\lambda} G_n(t,\mx,\lambda) + \Div_\mx P_n(t,\mx,\lambda) \qquad \hbox{in} \ {\cal D}'(\R^+\times \R^{d+1})\;,
\end{split}
\end{equation} where $(h_n)$ is a sequence bounded in $L^p(\R^+\times \R^{d+1})$, $p\geq 2$, while $(G_n)$ and $(P_n)$ are sequences which we specify later (see conditions d) and e) given at the beginning of Section \ref{sec:velocity_averaging}). The equation is linear (in $h_n$), but the right-hand side is singular 
(e.g.~it contains distributional derivatives of measures). Therefore, it is to be expected that the sequence of solutions $(h_n)$ to \eqref{transport} does not show regularity properties in the sense that it is strongly precompact in $L^p(\R^+\times \R^d \times \R)$. However, under the non-degeneracy conditions \eqref{lp-111} (with $a\equiv 0$) the sequence of averaged quantities $\int_\R \rho(\lambda) h_n(t,\mx,\lambda) d\lambda$, $\rho \in C^1_c(\R)$ is fixed, actually converges along a subsequence in $L^1_{loc}(\R^+\times \R^d)$. Moreover, a more specific result can be obtained under a quantitative variant of the non-degeneracy condition (see Theorem B below). In this case, the sequence of averages is bounded in $W^{s,r}_{loc}(\R^+\times \R^d)$ for some $s>0$ and $r\geq 1$. Several results in this direction can be found in \cite{EMM, nGPLS, Per, TT} and many others (see also references therein and the introduction of \cite{EMM} where a comprehensive overview on the subject can be found).

A generalisation of \eqref{transport} is given by the diffusive transport equation (i.e.~transport equation which also involves diffusion effects), having the form
\begin{equation}
\label{diffusive}
\begin{split}
\pa_t h_n+\Div_\mx \bigl(f(t,\mx,\lambda) h_n \bigr) =& \,\Div_\mx \bigl(\Div_\mx \left( a(t,\mx,\lambda) h_n \right)\bigr) \\
&\qquad + \pa_{\lambda} G_n(t,\mx,\lambda) + \Div_\mx P_n(t,\mx,\lambda) \qquad \hbox{in} \ {\cal D}'(\R^+\times\R^{d+1}) \,.
\end{split}
\end{equation} 
Despite the fact that it is physically more realistic (since it supposes to involve a standard physical phenomenon -- the diffusion), there are much less results regarding the velocity averaging of solutions to \eqref{diffusive}. 

Let us start this concise overview with \cite[Theorem C]{LPT}, where the authors formulated the velocity averaging result for \eqref{diffusive} with $(t,\mx)$-independent coefficients. However, the results is left without a proof and the proof cannot be found in any of the subsequent publications on this theme. 
Recently, in \cite{EMM}, the proof is provided for the $L^2$-framework and the diffusion satisfying condition b) above (with a slightly more regular matrix functions). 
On the other hand, in \cite{TT}, as well as in \cite{GH}, one can find a velocity averaging result for \eqref{diffusive}, still with $(t,\mx)$-independent coefficients, but under special non-degeneracy conditions. Roughly speaking, the conditions involve growth rate assumptions of the symbol as well as $\lambda$-derivative of the symbol of \eqref{diffusive} on the Littlewood-Paley dyadic blocks (see \cite[(2.3)]{GH} and \cite[(2.19), (2.20)]{TT}). Beside the fact that latter conditions are more restrictive than \eqref{lp-111}, it is complicated to check whether they are fulfilled. Here, we shall show that the regularity result for homogeneous degenerate parabolic equations holds without \cite[(2.20)]{TT} (see Theorem \ref{t-regularity} and Remark \ref{rem:non-deg}(i)).

%We also note that the procedure from \cite{GH} is (a highly non-trivial) extension of the results from \cite{TT} in the stochastic setting.

However, in both approaches given in \cite{GH,TT} and \cite{EMM} the fact that the diffusion matrix $a$ is 
positive semi-definite has been only mildly used. 
For instance, in the latter work we could also assume that $a=
  \left(\begin{array}{@{}c|c@{}}
    0 & 0 \\
    \hline
    0 & -\sigma \sigma^{_T} 
  \end{array}\right)$ (i.e.~that $a$ is negative semi-definite which should not demonstrate regularisation effects) and the same proof would work. The procedures described in \cite{GH, TT} present an even more general scenario, as the positive semi-definiteness of the matrix $a$ is only required in the proof of the so-called truncation property (refer to \cite[Section 2.4]{TT}). Consequently, no explicit sign constraint on $a$ is necessary, as demonstrated by examples such as $\operatorname{diag}(\lambda,\lambda^3)$, where $\lambda \in \R$ (which satisfies the truncation property). Therefore, the regularising effects that are typically expected from the diffusive term in \eqref{d-p} are not recast by the equation \eqref{diffusive} in the findings of \cite{EMM,GH,TT}. Consequently, improved regularity for the averaged quantities $\int_\R \rho(\lambda) h_n(t,\mx,\lambda) d\lambda$, $n\in \N$, should not generally be anticipated. In fact, for larger $t$, the potential presence of negative diffusion could lead to concentration effects and compromise the solution regularity.

Keeping this in mind, we introduce here a kinetic formulation of \eqref{d-p} which takes into account smoothing effects of the diffusion. 
The new conditions are given in f) below and they enable us to prove an appropriate velocity averaging lemma for \eqref{diffusive}. It can be shown (fairly directly from \cite{CK}) that a sequence of approximate solutions to \eqref{d-p} satisfies a kinetic equation whose properties are such that we can apply the new velocity averaging result. This leads us to an important application of the velocity averaging lemma -- existence of a weak solution to \eqref{d-p} given by Theorem A. The result represents a significant improvement with respect to the previous contributions since we allow explicit dependence on $(t,\mx)$ variables in both convection and diffusion terms (while in the convection term dependence might be even discontinuous).

In the final part of the paper, we address the question of regularity of entropy solutions to \eqref{d-p}, in the sense of \cite[Definition 2.2]{CP}, when the coefficients are $(t,\mx)$--independent. Again, to avoid unnecessary technicalities in the introduction, we shall informally formulate the second main result of the paper. The precise formulation is given in Section 5.

\mybox{
{\bf Theorem B.} Let $I\subseteq \R$ be a compact interval. Assume that the coefficients $\mff$ and $a$ from equation \eqref{d-p} are $(t,\mx)$-independent and satisfy {\rm a), b)} and the non-degeneracy condition:
\begin{equation}
\label{non-deg-stand}
(\forall\delta\in (0,\delta_0))\quad \sup\limits_{|(\xi_0,\mxi)|= 1} 
	\operatorname{meas}\Bigl\{\lambda \in I :\, \big|\xi_0+\langle \mff'(\lambda)\,|\,{\mxi}
	\rangle \big|^2+\langle a(\lambda){\mxi} \,|\,{\mxi} \rangle \leq \delta \Bigr\}\lesssim \delta^{\alpha}\,,
\end{equation}  
for some $\alpha, \delta_0>0$.

Then, the entropy solution to \eqref{d-p} belongs to the Sobolev space $W^{s,q}_{loc}(\R^+\times \R^d)$ for some constants $s>0$ and $q>1$.
}

The result of the latter theorem seems to be of particular interest and it has attracted significant attention for the last three decades. We note that it is also explicitly posed in \cite[page 1509]{TT} (how to use the entropy conditions to achieve an (improved) regularity result for degenerate parabolic equations). In the scenario where the diffusion equals zero, we encounter a scalar conservation law. In such cases, the regularity results are available across a broad time span. We mention some of the first results \cite{Ago, LPT} and a couple of recent results \cite{GL, Syl} where one can find numerous other references. 
When it comes to the case when the diffusion is not equal to zero, we are not aware of the result of the form given in Theorem B. We have already mentioned \cite{GH, TT} where the latter result holds under conditions which are more restrictive than the one of Theorem B, and they are checked only in special cases of \eqref{d-p} 
(see \cite[Subsection 2.4]{GH} and \cite[Corollaries 4.1--4.5]{TT}). In fact, using Theorem B we are able to improve the result of \cite[Corollary 4.5]{TT} and prove that the example demonstrates the regularisation effects for any $n,l$ (see Corollary \ref{TT-ex}).
As an important example of \eqref{d-p}, we have the porous media equation and the recent (optimal) regularity result given in \cite{Gess, GST} (see also references therein).

The paper is organised as follows. After the Introduction, we recall the necessary properties of the basic tools we are using -- H-measures, and at the same time we develop some new results. These results are then used in Section 3 to prove a velocity averaging result for diffusive transport equations, which is the key ingredient in the analysis carried out in the following section. 
Moreover, the techniques developed in Section 3 are the basis for deriving the results of Section 5.
In Section 4, we prove existence of weak solutions to Cauchy problems corresponding to \eqref{d-p} and in Section 5, we obtain the regularity result of entropy solutions for the autonomous (homogeneous) variant of \eqref{d-p}.

%Unlike the approach from \cite{AlM, EM} where we needed to use a scaled blow up approach or \cite{Frid, Kwon} where the equation is essentially such that it was possible split the 

\section{Multiplier operators and the H-measures}\label{sec:multipliers}

Micro-local defect measures (or H-measures) are discovered simultaneously by P.~G\'erard \cite{Ger} and L.~Tartar \cite{Tar}.
They describe loss of the strong $L^2$-compactness of sequences weakly converging in $L^2$. In other words, to each sequence weakly converging to zero in $L^2$, we can adjoin a H-measure. If the H-measure is zero, then the sequence in fact converges strongly. The latter notion was a step forward with respect to the defect measures by P.-L.~Lions \cite{Evans} since the H-measures involve dual variables as well. 

They appeared to be very useful in the velocity averaging results  \cite{EMM, Ger, LM2}, existence of traces and solutions to nonlinear evolution equations \cite{1, EM2022, HKMP, pan_jhde}, 
compensated compactness results to equations with variable coefficients \cite{Ger, MM, pan_aihp, Tar}, in homogenisation (from where the letter H in the title of the measure stems) \cite{ALjmaa, tar_book}, etc. 

Moreover, this has initiated a variety of different generalisations to the original micro-local defect measures which we call here micro-local defect functionals.
We mention parabolic, ultra-parabolic and adaptive variants of the H-measures \cite{ALjfa, EMM, pan_aihp}, H-measures as duals of Bochner spaces \cite{LM2},
H-distributions \cite{AEM, AM, LM3, MM}, micro-local compactness forms \cite{rindler}, variants with a scale \cite{AE, AEL, Tar2}, etc.

We denote by $\mx=(x_1,x_2,\dots,x_d)$ points (vectors) in $\Rd$ and by $\mxi=(\xi_1,\xi_2,\dots,\xi_d)$ we denote the corresponding dual variables in the sense of the Fourier transform (if $t$ occurs as above, then we use $\xi_0$ for the dual variable). 

We define the Fourier transform by $\hat{u}(\mxi) = {\cal F}u(\mxi) = \int_{\Rd} e^{-2\pi i \langle\mxi\,|\,\mx\rangle}u(\mx)\,d\mx$, and its inverse by $(u)^\vee(\mx) = \bar{\cal F}u(\mx) = \int_{\Rd} e^{2\pi i \langle\mxi\,|\,\mx\rangle}u(\mxi)\,d\mx$ (here $\langle\cdot\,|\,\cdot\rangle$ denotes the Euclidean scalar product on $\Rd$, while the same notation is used for $\mathbb{C}^d$ when we take it to be antilinear in the second argument). 
The Fourier multiplier operator is then defined by ${\cal A}_\psi u= (\psi\hat u)^\vee$. 

If ${\cal A}_\psi$ is bounded on ${L}^p(\Rd)$, $p\in (1,\infty)$, i.e.~as a linear map from ${L}^p(\Rd)$ to ${L}^p(\Rd)$, 
it is called the $L^p$-Fourier multiplier operator and the symbol $\psi$ we refer to as the 
$L^p$-Fourier multiplier. If, besides of $\mxi$, the symbol $\psi$ depends on $\lambda$ as well, then $\lambda$ is considered as a parameter.

A complete characterisation of $L^p$-Fourier multipliers for general $p\in (1,\infty)$ is still missing. 
However, in the case of $p\in \{1,2\}$ the situation is fully resolved \cite[theorems 2.5.8 and 2.5.10]{Gra}:
\begin{theorem}
\label{L1-multiplier}
Let $\psi:\Rd\to\C$ be a function and $\mathcal{A}_\psi$ the Fourier multiplier operator.
\begin{itemize}
\item[i)] ${\cal A}_{\psi}$ is an $L^1$-multiplier operator if and only if ${\cal F}^{-1}(\psi)$ is a finite Borel measure, and the norm is equal to the total variation of the measure ${\cal F}^{-1}(\psi)$.
If in addition ${\cal F}^{-1}(\psi)\in L^1(\Rd)$, then it holds $\|{\cal A}_{\psi}\|_{L^1\to L^1}=\|{\cal F}^{-1}(\psi)\|_{L^1(\Rd)}$.
\item[ii)] ${\cal A}_{\psi}$ is an $L^2$-multiplier operator if and only if $\psi$ is a bounded function, 
and the norm is equal to  $\|\psi\|_{L^\infty(\Rd)}$.
\end{itemize}
\end{theorem}

\begin{remark}\label{rem:L1Linfty}
Since any $L^1$-multiplier operator is an $L^\infty$-multiplier operator (the converse does not hold; cf.~\cite[Example 2.5.9]{Gra}), with the same norms,
under the assumption ${\cal F}^{-1}(\psi)\in L^1(\Rd)$ we also have 
$\|{\cal A}_{\psi}\|_{L^\infty\to L^\infty}=\|{\cal F}^{-1}(\psi)\|_{L^1(\Rd)}$.
\end{remark}

\begin{remark}\label{rem:Cd+1}
If $\psi\in C_c^{d+1}(\Rd)$, 
then ${\cal F}^{-1}\psi\in L^1(\Rd)$ and
$\|{\cal F}^{-1}\psi\|_{L^1(\Rd)} \lesssim \|\psi\|_{C^{d+1}(\Rd)}$ (in the paper we us the notation 
$\lesssim X$ in the standard way, meaning that  $\leq C X$ for a constant $C$ 
independent of parameters appearing in $X$).
Indeed, the claim follows by the fact that $(1+|\mxi|^2)^{-\frac{d+1}{2}}\in L^1(\Rd)$ and the estimate
$$
\|(1+|\mxi|^2)^{\frac{d+1}{2}}{\cal F}^{-1}\psi\|\lesssim 
	\max_{|\malpha|\leq d+1}\|\partial^\malpha\psi\|_{L^1(\Rd)} 
	\lesssim \|\psi\|_{C^{d+1}(\Rd)} \,,
$$
where we have used that each $\partial^\malpha\psi$,
$|\malpha|\leq d+1$, is bounded with compact support.
\end{remark}

In order to prove the $L^p$ boundedness in the case $p\in(1,\infty)\setminus\{2\}$, 
we are left only with several available sufficient conditions.
Here we present a corollary of the Marcinkiewicz multiplier theorem \cite[Corollary 5.2.5]{Gra}:
\begin{theorem}
	\label{thm:m1} Suppose that  $\psi\in C^{d}(\R^d\setminus\cup_{j=1}^d\{\xi_j= 0\})$ is a bounded function such that for some
	constant $C>0$ it holds
	\begin{equation*}
%	\label{c-mar} 
	|\mxi^{\malpha} \partial^{\malpha}
	\psi(\mxi)|\leq C,\ \
	\mxi\in \R^d\backslash \cup_{j=1}^d\{\xi_j= 0\}
	\end{equation*}
	for  every multi-index
	$\malpha=(\alpha_1,\dots,\alpha_d) \in {\N}_0^d$
	such that
	$|\malpha|=\alpha_1+\alpha_2+\dots+\alpha_d
	\leq d$. Then $\psi$ is an $L^p$-multiplier for
	any $p\in (1,\infty)$, and the operator norm of ${\cal A}_\psi$ equals 
	$C_{d,p} C$, where
	$C_{d,p}$ depends only on $p$ and $d$.
\end{theorem}

We turn now to micro-local defect functionals that we shall use in the paper.
Let us first start with a brief overview of existing tools and approaches. 
The scaling in the classical H-measure with respect to all variables is the same, due to the projection
$\mxi\mapsto \mxi/|\mxi|$ (see Theorem \ref{bilinearboundedness} below). As a consequence, that means 
the H-measure (and corresponding variants) are suitable for differential relations where the ratio of 
the highest orders of derivatives in each variable is the same (cf.~\cite{ALjmaa}). 
Since that is not the situation we have in \eqref{d-p}, several variants and approaches were developed. 
Let us mention ultra-parabolic H-measures \cite{LM2, pan_aihp} for which the diffusion matrix $a$
is allowed to degenerate, but only in fixed directions. This assumption can be removed with adaptive H-measures
\cite{EMM}, however, they can be applied only for homogeneous $a$ (independent of $t$ and $\mx$).
Here we develop a different approach where we replace the equation \eqref{d-p} with two in which flux and diffusion 
are dominated, respectively. In that way we can keep the same scaling with respect to all variables, and 
the price to pay is that we have an additional equation to analyse.
Therefore, we are able to use an already developed variant (Theorem \ref{bilinearboundedness} below), 
but which we need to slightly generalise due to singular terms occurring in the auxiliary equation 
(Theorem \ref{H-for-diffusion} below). 

In what follows, $\mathcal{M}_{(loc)}(S)$
denotes the space of (locally) bounded Radon measures on $S\subseteq  \Rd$, and $\|\cdot\|_{\mathcal{M}(S)}$ stands 
for the total variation of a bounded Radon measure on $S$. 
By $L^2_{w*}(K;\mathcal{M}(S))$ we denote the dual of $L^2(K;C_0(S))$, which 
is a Banach space of weakly $*$ measurable functions $\mu:K\to \mathcal{M}(S)$
such that $\int_K \|\mu(\lambda,\cdot)\|^2_{\mathcal{M}(S)} d\lambda<\infty$.

% THEOREM BEGIN
\begin{theorem}\cite[Theorem 2.2]{LM5}\label{bilinearboundedness}
	Let $(u_n(\cdot_\mx,\cdot_\lambda))$ be bounded in ${L}^q(\R^d\times\R)$, for some $q> 2$, and 
	uniformly compactly supported on $\Omega \times K \subset\subset \R^d\times \R$.
	%weakly converging to zero in ${\rm L}^q(\R^d\times\R)$, for some $q> 2$.
	Let $(v_n(\cdot_\mx))$ be an $\Omega$--uniformly compactly supported sequence weakly-$\star$ converging to zero in
	${L}^\infty(\R^d)$.
	
	Then there exists a subsequence (not relabelled) 
	and a continuous functional $\mu\in L^{2}_{w*}(K;{\cal M}(\Omega\times \Sdmj))$ such that for every
	$\varphi\in {L}^{\frac{2q}{q-2}}(\Omega\times K)$ and $\psi\in{C}^d(\Sdmj)$ it holds
	\begin{equation}
	\label{mu-repr-h}
	%\begin{split}
	\mu(\varphi \psi) = \lim_{n\to\infty}\int_{\Omega\times K}\varphi(\mx,\lambda) u_n(\mx,\lambda)\overline{{\cal A}_{\psi(\mxi/|\mxi|)}(v_n)(\mx)}\; d\mx \,d\lambda \,.
	%&= \lim_{n\to\infty}\int_{\R^d \times K} \bigl(\widehat{\varphi 
	%	u_n(\cdot,\lambda)}\bigr)(\mxi)\psi(\pi_P(\mxi,\lambda),\lambda) \overline{\hat{v}_n(\mxi)} 
	%	\; d\mxi \,d\lambda\;,
	%\end{split}
	\end{equation} 
	%where $\hat{}$ denotes the Fourier transform with respect to $\mx$.
	Furthermore, the functional $\mu$ has the representation
	\begin{equation}
	\label{repr}
	\mu(\mx,\lambda,\mxi)=g(\mx,\lambda,\mxi) d\nu(\mx,\mxi) \,, 
	\end{equation} 
	where $g\in L^2\bigl(K;L^1(\Omega\times \Sdmj:\nu)\bigr)$ and $\nu\in\mathcal{M}(\Omega,\times\Sdmj)$ is a non-negative, bounded, scalar Radon measure. 
	Moreover, $\Sdmj$-projection of the measure $\nu$ is absolutely continuous with respect to the Lebesgue measure 
	$d\mx$, i.e.~there exists an $h\in L^{\frac{q}{q-2}}(\Omega)$ such that 
	$\int_{\Sdmj} d\nu(\mx,\mxi)=h(\mx)d\mx$.
\end{theorem}

The representation \eqref{repr} is obtain through the method of slicing of measures and the Radon-Nikodym theorem 
(see e.g.~\cite{Evans}).

To deal with the diffusion term, we shall need the following variant of the H-measures, which can be seen 
as a generalisation (in a particular case) of \cite[Theorem 2.2]{LM5} in which the sequence 
$(u_n)$ is allowed to be measure-valued.

\begin{theorem}
\label{H-for-diffusion}
Assume that $(\beta_n)$ is a sequence bounded in $L^2(\Omega;{\cal M}(K))$ for some $\Omega \subset\subset \R^d$ and $K\subset\subset \R$ in the sense that
\begin{equation}
\label{beta-bnd}
\int_\Omega \|\beta_n(\mx,\cdot)\|^2_{{\cal M}(K)} d\mx<c_\beta^2 \,,
\end{equation} for some constant $c_\beta>0$. Let $(v_n)$ be a bounded (by a constant $c_v>0$), $\Omega$--uniformly compactly supported sequence in ${L}^\infty(\Rd)$, converging to 0 in the weak--$\star$ sense.

Then, there exist subsequences (not relabelled) $(\beta_n)$ and $(v_n)$ 
and a measure $\mu\in{\cal M}(\R\times \Omega\times \Sdmj)$ such that for every
$\varphi\in C_0(\Omega)$, $\rho\in C_0(K)$, and $\psi\in{C}(\Sdmj)$ it holds
\begin{equation*}
%\label{mu-d}
%\begin{split}
\mu(\varphi \rho \psi ) = \lim_{n\to\infty}\int_{\Omega}\varphi(\mx) \int_K \rho(\lambda) d\beta_n(\mx,\lambda) \, \overline{{\cal A}_{\psi(\frac{\cdot}{|\cdot|})}(v_n)(\mx)}\; d\mx  \,.
%\end{split}
\end{equation*} 
\end{theorem}

\begin{proof}
Let us consider a sequence of functionals $(\mu_n)$ defined on $L^\infty(\Omega)\times C_0(K) \times C(\Sdmj)$ via
\begin{align*}
\langle \mu_n, \varphi\otimes \rho \otimes \psi \rangle:=\int_{\Omega} \varphi(\mx) \int_K \rho(\lambda) d\beta_n(\mx,\lambda) \, {\cal A}_{\psi(\frac{\cdot}{|\cdot|})}(v_n)(\mx) d\mx \,,
\end{align*} 
and extended by linearity to finite linear combinations. 
The strategy of the proof is to show that each $\mu_n$ can be extended to the functional on $C_0(\R\times\Omega\times\Sdmj)$, 
where they form a bounded sequence of measures on $\Omega \times K \times \Sdmj$. Then the result follows by applying the Banach-Alaoglu theorem on $(\mu_n)$ in the space 
$\mathcal{M}(\Omega \times K \times\Sdmj)=\bigl(C_0(\Omega\times K\times\Sdmj)\bigr)'$. 

To this end, we note that any $\phi\in C_0(\Omega\times K \times \Sdmj)$ can be uniformly approximated by functions of the form
$$
\phi_N(\mx,\lambda,\mxi):= \sum\limits_{j=1}^N \chi_j(\mx)\rho_j(\lambda)\psi_j(\mxi)  \,, \ \ \rho_j\in C_c(K)\,, \; \psi_j \in C^d(\Sdmj)\,,
$$
where $\chi_j$, $j=1,2,\dots, N$, are characteristic functions of mutually disjoint sets.
By the density argument, it is enough to show that $\langle\mu_n,\phi_N\rangle$ is bounded by $\|\phi_N\|_{L^\infty(\Omega\times K \times \Sdmj)}$ times a constant independent of $n$ and $N$.

%Let us start by noting that the sequence $(\mu_n)$ is clearly bounded on $L^\infty(\Omega)\times C_0(K) \times C^d(\Sdmj)$ since
%\begin{align*}
%\big|\langle \mu_n, &\varphi\otimes \rho \otimes \psi \rangle \big|\\
%&\leq \Bigl\| \varphi(\cdot) \int_K \rho(\lambda) d\beta_n(\cdot,\lambda)\Bigr\|_{L^2(\Omega)} 
%\bigl\| {\cal A}_{\psi(\mxi/|\mxi|)}(v_n)(\cdot) \bigr\|_{L^2(\Omega)}
%\\ & \leq {c_\beta} c_v \operatorname{meas}(\Omega)^{1/2}\, \|\varphi\|_{L^\infty(\Omega)} \|\rho\|_{C_0(K)} \|\psi\|_{C(\Sdmj)}\,,  
%\end{align*}
%where we have used that $\mxi\mapsto\psi(\mxi/|\mxi|)$ is bounded by $\|\psi\|_{C(\Sdmj)}$, implying that 
%$\mathcal{A}_{\psi(\frac{\cdot}{|\cdot|})}$ is an $L^2$-Fourier multiplier operator with the norm equal to
%$\|\psi\|_{C(\Sdmj)}$.

We have
\begin{align*}
\langle \mu_n, \phi_N  \rangle
	&= \int_{\Omega} \sum\limits_{j=1}^N \chi_j(\mx) \int_K \rho_j(\lambda) d\beta_n(\mx,\lambda) 
	\, {\cal A}_{\psi_j(\frac{\cdot}{|\cdot|})}(v_n)(\mx) d\mx
\\&=\int_{\Omega} \sum\limits_{j=1}^N \chi_j(\mx) \int_K \rho_j(\lambda) d\beta_n(\mx,\lambda) \, {\cal A}_{\psi_j(\frac{\cdot}{|\cdot|})}(\chi_j v_n)(\mx) d\mx \,,
\end{align*}
where in the last step we used the commutation lemma (Lemma \ref{izclankaokomutatoru} below; see also Theorem \ref{thm:m1}) together with the fact that $\chi_j=\chi_j^2$.
From here, using the Plancherel formula, the fact that $\chi_j \chi_k=0$ for $i\neq k$, and the discrete variant of the Cauchy-Schwartz inequality, we have
\begin{align*}
&\langle \mu_n, \phi_N  \rangle \\ 
& \leq \int\limits_{\Omega} \left( \sum\limits_{j=1}^N \Bigl|\psi_j\bigl(\frac{\mxi}{|\mxi|}\bigr) {\cal F}_{\mx} \Bigl(\chi_j(\cdot) \int_K \rho_j(\lambda) d\beta_n(\cdot,\lambda)\Bigr)(\mxi)\Big|^2 \right)^{1/2} \, \left(\sum\limits_{j=1}^N \big|{\cal F}_{\mx}(\chi_j v_n)(\mxi)\big|^2 \right)^{1/2}  d\mxi \\
& \leq \left(\sum\limits_{j=1}^N \|\psi_j\|_{C(\Sdmj)}^2 \Bigl\|\chi_j(\cdot ) \int_K \rho_j(\lambda) d\beta_n(\cdot,\lambda)\Bigr\|_{L^2(\Omega)}^2 \right)^{1/2}\, \left(\sum\limits_{j=1}^N\|  \chi_j v_n\|_{L^2(\Omega)}^2\right)^{1/2}\\
& \leq \Biggl\|\sum\limits_{j=1}^N \|\psi_j\|_{C(\Sdmj)} \chi_j(\cdot ) \int_K \rho_j(\lambda) d\beta_n(\cdot,\lambda)\Biggr\|_{L^2(\Omega)} \, \biggl\| \sum\limits_{j=1}^N \chi_j v_n\biggr\|_{L^2(\Omega)}\\
& \leq  \biggl\|\sum\limits_{j=1}^N \|\psi_j\|_{C(\Sdmj)} \,  \| \rho_j \|_{C_0(K)} \, \|\beta_n(\cdot_{\mx},\cdot_\lambda)\|_{{\cal M}(K)} \, \chi_j(\cdot )\biggr\|_{L^2(\Omega)} 
	\, \| v_n\|_{L^2(\Omega)}\\
& \leq  c_\beta c_v \operatorname{meas}(\Omega)^{1/2} \, \big\| \sum\limits_{j=1}^N  \chi_j  \, \rho_j \, \psi_j \big\|_{L^\infty(\Omega\times K \times \Sdmj)} \,,
\end{align*} 
where in the last step we used
$$
\sum\limits_{j=1}^N \|\psi_j\|_{C(\Sdmj)} \,  \|\rho_j\|_{C(K)} \, \chi_j(\mx)= \Bigl\|\sum\limits_{j=1}^N \chi_j({\mx}) \rho_j(\cdot_{\lambda}) \psi_j(\cdot_{\mxi}) \Bigr\|_{C(K\times \Sdmj)} \,.
$$ 
Therefore, the density argument provides that each functional $\mu_n$ can be extended 
on $C_0(\Omega\times K \times \Sdmj)$ and 
$\|\mu_n\|_{\mathcal{M}(\Omega\times K\times\Sdmj)}\leq c_\beta c_v\operatorname{meas}(\Omega)^{1/2}$. 
Thus, $(\mu_n)$ is a bounded sequence of measures on $\Omega\times K\times\Sdmj$, 
which should have been shown. 
\end{proof}

A variant of the First commutation lemma (which is given in \cite[Lemma 1]{1stcommlemm}; see also Remark 2 in the mentioned reference) reads:

\begin{lemma}\label{izclankaokomutatoru}
	Let $(v_n)$ be a bounded, uniformly compactly supported sequence in ${L}^\infty(\Rd)$, converging to 0 in the sense of
	distributions, and let $\psi\in C^d(\Sdmj)$.
%	Let $\psi\in{C}^d(\Rd\!\setminus\!\{0\})\cap{L}^\infty(\Rd)$ be an ${L}^p$-multiplier, 
%	for any $p\in\oi 1\infty$, which satisfies
%	\begin{equation}\label{eq:compactness_cond}
%	\lim_{|\mxi|\to\infty}\sup_{|{\bf h}|\leq1}\left| \psi(\mxi+{\bf h}) - \psi(\mxi) \right|=0 \;.
%	\end{equation}
	
	Then for any $b\in{L}^\infty(\Rd)$ and $r\in\oi1\infty$ the following holds:
	$$
	b{\cal A}_{\psi(\frac{\cdot}{|\cdot|})}(v_n)-{\cal A}_{\psi(\frac{\cdot}{|\cdot|})}(b v_n) \longrightarrow 0 \quad
	\hbox{strongly in} \quad {L}^r_{loc}(\Rd) \;.
	$$
\end{lemma}

We will also need the following estimate involving measure-valued functions 
of the type appearing in the previous theorem.

\begin{lemma}
	\label{L-ocjena-L1} 
	Let $\psi\in C(\R;C_c^{d+1}(\R^d))$ and $\rho \in C_c(\R)$. Let $\beta\in L_c^1(\R^d;{\cal M}(\R))$. 
	Denote
	\begin{align*}
	&\int_{\R}\Bigl({\cal A}_{ \psi(\mxi,\lambda)} \rho(\lambda)\,d\beta(\cdot,\lambda)\Bigr)(\mx)\\
	&\quad :=
	\int_{\R^{2d}} e^{2\pi i (\mx-\my)\cdot \mxi}\int_{\R}\psi(\mxi , \lambda) \rho(\lambda)\,d\beta(\my,\lambda)  \, d\my\,d\mxi
	={\cal F}^{-1}\Big(\int_{\R} \psi(\cdot,\lambda) \rho(\lambda)\,d\hat{\beta}(\cdot,\lambda) \Bigr)(\mx) \,,
	\end{align*}
	where $d\hat{\beta}(\cdot,\lambda)$ is the Fourier transform of $\beta$ with respect to the first variable.

	Then for every $\rho\in C_c(\R)$ it holds:
	\begin{equation*}
%	\label{ocjena-L1}
	\begin{split}
	\Bigl\|\int_{\R}{\cal A}_{ \psi(\mxi,\lambda)} & \rho(\lambda)\,d\beta  \Bigr\|_{L^1(\R^d)} \\& 
	\lesssim_{\supp(\rho\psi)} \| \psi\|_{C(\supp\rho;C^{d+1}(\R^{d}))}  \,  \Bigl\| \int_{\R} |\rho(\lambda)|\,d|\beta|  \Bigr\|_{L^1(\R^d)} \,,
	\end{split}
	\end{equation*}
	where, for any $\mx\in\Rd$, $|\beta|=|\beta|(\mx,\cdot)$ denotes the total variation measure of $\beta=\beta(\mx,\cdot)$.
\end{lemma}
\begin{proof}
	For a fixed $\psi\in C(\R;C^{d+1}_c(\R^{d}))$, we approximate it in $L_{loc}^\infty(\R;C_c^{d+1}(\R^d))$ trough the set of disjoint intervals $(K_j^n)_{j=1,\dots,n}$ with the length of order $\frac{1}{n}$, partitioning the support of $\rho$, by the function 
	$$
	\psi_n(\mxi,\lambda)=\sum\limits_{j=1}^n \psi(\mxi,\lambda_j) \chi^n_j(\lambda) 
	$$ 
	where $\chi^n_j$ is the characteristic function of the interval $K^n_j$.
	
	Since ${\cal A}_{\psi_n(\cdot,\lambda)}$ is for any $\lambda\in\R$
	an $L^1$-Fourier multiplier operator (see Theorem \ref{L1-multiplier}
	and Remark \ref{rem:Cd+1}), it holds
	\begin{align*}
	\Bigl\|\int_{\R}{\cal A}_{\psi_n(\mxi,\lambda)} & \rho(\lambda)\,d\beta \Bigr\|_{L^1(\R^d)}=
		\Bigl\|\sum\limits_{j=1}^n {\cal A}_{\psi(\mxi,\lambda_j)} \int_{\R} \chi^n_j(\lambda) \rho(\lambda)\,d\beta  \Bigr\|_{L^1(\R^d)}\\
	& \leq \sum\limits_{j=1}^n \bigl\| {\cal F}^{-1} \left(\psi(\mxi,\lambda_j)\right) \bigr\|_{L^1(\R^d)} \, 
	\Bigl\|\int_{\R} \chi^n_j(\lambda) \rho(\lambda)\,d\beta  \Bigr\|_{L^1(\R^d)}	\\
%	&=\sum\limits_{j=1}^n \bigl\| {\cal F}^{-1} \left(\psi(\mxi,\lambda_j)\right) \bigr\|_{L^1(\R^d)} 
%	\int_{\R^d}\int_{\R} \chi^n_j(\lambda) |\rho(\lambda)|\,d|\beta_k| \, d\mx\\
	&\leq \int_{\R^d}\int_{\R} \sum\limits_{j=1}^n \bigl\| {\cal F}^{-1} \left(\psi(\mxi,\lambda_j)\right) \bigr\|_{L^1(\R^d)}\chi^n_j(\lambda) |\rho(\lambda)|\,d|\beta| \, d\mx\\
	& \leq  \sup\limits_{\lambda \in \R} \Bigg|\sum\limits_{j=1}^n  \chi^n_j(\lambda) \bigl\| {\cal F}^{-1} \left(\psi(\mxi,\lambda_j)\right) \bigr\|_{L^1(\R^d)} \Biggr|\, \int_{\R^d}\int_{\R}  |\rho(\lambda)|\,d|\beta| \, d\mx
	\\&\lesssim \sup\limits_{\lambda \in \R} \Bigg|\sum\limits_{j=1}^n  \chi^n_j(\lambda) \bigl\|\psi(\mxi,\lambda_j)\bigr\|_{C^{d+1}(\R^d)} \Biggr|\, \int_{\R^d}\int_{\R}  |\rho(\lambda)|\,d|\beta| \, d\mx
	\end{align*} Letting $n\to \infty$ here, we reach to
	
	\begin{equation*}
	\begin{split}
	&\Bigl\|\int_{\R}{\cal A}_{\psi(\mxi,\lambda)} \rho(\lambda)\,d\beta \Bigr\|_{L^1(\R^d)} \lesssim \|\psi(\mxi,\lambda) \|_{C(\R_\lambda;C_c^{d+1}(\R^d))} \, \Bigl\| \int_{\R} |\rho(\lambda)|\,d|\beta|  \Bigr\|_{L^1(\R^d)} \,.
	\end{split}
	\end{equation*} This concludes the proof. \end{proof}

\begin{remark}
	\label{Lp-bound} We note that the same bound holds if we replace $d\beta$ by an arbitrary $L^1(\R^{d+1})$-bounded function with compact support.
\end{remark}

We shall also need the fact that an action of the multiplier operator can be defined on the space of Radon measures. The following lemma holds.
\begin{lemma}
	\label{l-meas}
	Assume that $\psi:\R^d\to \C$ is such that ${\cal F}^{-1}(\bar\psi)\in L^1(\R^d)$. Let $\mu\in {\cal M}(\R^d)$ be a bounded Radon measure. Then for the action of the multiplier operator ${\cal A}_\psi$ on the measure $\mu$ defined for all $\varphi \in C_0(\R^d)$ by
	$$
	\langle {\cal A}_{\psi}(\mu)\, , \, \bar\varphi \rangle := \langle \mu\, , \, \overline{{\cal A}_{\bar\psi}(\varphi)} \rangle=\int_{\R^d} \overline{{\cal A}_{\bar\psi}(\varphi)}\,d\mu   \,,
	$$ the following bound holds
	\begin{equation*}
	%\label{meas-mult-bnd}
	\|{\cal A}_{\psi}(\mu)\|_{{\cal M}(\R^d)} \lesssim \|{\cal F}^{-1}(\bar\psi)\|_{L^1(\R^d)} \,,
	\end{equation*}
	where the constant in the inequality depends on $d$ and $\mu$, 
	i.e.~it is independent of $\psi$. 
\end{lemma}
\begin{proof}
	First note that the function $\psi$ is a symbol of an $L^1$-multiplier operator (see Theorem \ref{L1-multiplier}). Thus, it is also a symbol of an $L^\infty$ multiplier operator with the same norm (see Remark \ref{rem:L1Linfty}). Therefore, we have 
	\begin{equation*}
	%\label{mu-est-1}
	\begin{split}
	\bigl|\langle {\cal A}_{\psi}(\mu)\,, \, \varphi \rangle\bigr| = \Bigl|\int_{\R^d} \overline{{\cal A}_{\bar\psi}(\varphi)}\,d\mu \Bigr|  
	&\leq  \|{\cal A}_{\bar\psi}(\varphi)\|_{L^\infty(\R^d)} \, \|\mu\|_{{\cal M}(\R^d)} \\&\leq \|\mu\|_{{\cal M}(\R^d)} \|{\cal F}^{-1}(\bar\psi)\|_{L^1(\R^d)} \|\varphi\|_{L^\infty(\R^d)} \,,
	\end{split}
	\end{equation*} i.e. 
	$$
	\|{\cal A}_{\psi}(\mu) \|_{{\cal M}(\R^d)} \leq \|\mu\|_{{\cal M}(\R^d)} \|{\cal F}^{-1}(\psi)\|_{L^1(\R^d)}.
	$$ This concludes the proof.
\end{proof}

\begin{remark}\label{l-meas-rem}
Analogously as in the proof of Lemma \ref{L-ocjena-L1}, 
we can generalise the statement of the previous lemma 
to the situation where $\mu\in L^1(\R_\lambda;\mathcal{M}(\Rd))$
and $\psi\in C_c(\R_\lambda; C^{d+1}_c(\Rd))$, obtaining
\begin{equation*}
\biggl\|\mathcal{F}^{-1}\Bigl(\int_\R \psi(\cdot,\lambda)\, d\hat\mu (\cdot,\lambda)d\lambda\Bigr)\biggr\|_{\mathcal{M}(\Rd)}
	= \biggl\|\int_\R\mathcal{A}_{\psi(\cdot,\lambda)}(\mu(\cdot,\lambda))
	\,d\lambda\biggr\|_{\mathcal{M}(\Rd)}
	\lesssim_\mu \|\psi\|_{C(\R;C^{d+1}(\Rd))} \,.
\end{equation*}
\end{remark}

%========================================================================
%===============================================================================
\section{A velocity averaging result}\label{sec:velocity_averaging}
% =========================================================================

In this section, we shall prove a velocity averaging result for sequences of degenerate parabolic transport equations (diffusive transport equations). We shell not distinguish between the time and space variables which makes the result more general.
\begin{equation}
\label{TEq}
\begin{split}
\sum\limits_{k=1}^d \pa_{x_k}\Bigl( \bigl(f_k(\mx,\lambda)&+ \sum\limits_{j=1}^d a_{kj,x_j}(\mx,\lambda)\bigr) h_n \Bigr) = \,\sum\limits_{k,j=1}^d \pa^2_{x_k x_j}\left( a_{kj}(\mx,\lambda) h_n(\mx,\lambda) \right)\bigr) \\
&+ \pa_{\lambda} G_n(\mx,\lambda) + \Div_\mx P_n(\mx,\lambda) \quad \hbox{in} \ {\cal D}'(\R^{d+1})\;.
\end{split} 
\end{equation} By putting
$$
\bigl(f_1(\mx,\lambda),\dots,f_d(\mx,\lambda)\bigr)=f(\mx,\lambda):=\pa_\lambda \mff(\mx,\lambda) 
$$ 
we assume that conditions a)\,--\,c) are fulfilled (while keeping in mind that we incorporated the time variable in the space variables). 
Here we used a shorthand $a_{kj,x_j}=\partial_{x_j}a_{kj}$, 
which will be repeatedly used throughout this and the following section.
Compared to \eqref{diffusive}, here we have written the first order term on the left-hand side in a more precise form taking into account the contribution coming from the heterogeneity of the diffusion matrix $a$.
By formal calculation from two terms with the diffusion matrix, we obtain the expression $\Div_\mx(a(\mx,\lambda)\nabla_\mx h_n(\mx,\lambda))$ on the right side, which indeed corresponds to a nonhomogeneous diffusion. 
This formal approach is justified by the assumption f) below.

Regarding the terms $(G_n)$ and $(P_n)$, we assume

\begin{itemize}

\item[d)] $G_n \to G$ strongly in ${L}_{loc}^{r_0}(\R_\lambda; {W}_{loc}^{-1,r_0}(\R^{d}))$ for some 
	$r_0\in(1,\infty)$;

\item [e)] $P_n=(P_1^n,\dots,P_d^n)\to P$ strongly in ${L}_{loc}^{p_0}(\R^{d}\times\R_\lambda;\Rd)$ for some $p_0\in(1,\infty)$.

\end{itemize} Finally, we need to additionally assume the following property of the solutions $(h_n)$ which essentially entails expected regularisation effects due to diffusion:
\begin{itemize}
\item[f)] there exist sequences $(\beta_j^n)_{n\in \N}$, $j=1,2,\dots,d$, of functions bounded by $c_\beta>0$ in $L^2(\Omega;{\cal M}(K))$ for every $K\subset\subset \R$ and $\Omega\subset\subset \R^{d+1}_+$ in the sense of \eqref{beta-bnd} such that
\begin{equation}
\label{regularity}
\begin{split}
\sum\limits_{k,j=1}^d & \pa^2_{x_k x_j}\big( a_{kj}(\mx,\lambda) h_n(\mx,\lambda) \big)-\sum\limits_{k,j=1}^d \pa_{x_k} \bigl( a_{kj,x_j}(\mx,\lambda)h_n(\mx,\lambda)\bigr)\\
&= \sum\limits_{k,j=1}^d  \pa_{x_j} \Bigl(\sigma_{kj}(\mx,\lambda) \beta^n_k(\mx,\lambda)\Bigr) \,,
\end{split}
\end{equation} where $\sigma$ is the matrix valued function defined in b).

\end{itemize} %Due to linearity of the equation, we can assume that $(h_n)$ is uniformly compactly supported and that it tends weakly to zero which we actually do in the following theorem.

\begin{remark}
A motivation for the assumption f) comes from the entropy formulation given in
\cite{CK} (see Section \ref{sec:application} below and in particular Lemma \ref{condf}).
\end{remark}

We are now in position to state and prove the main result of this section,
the velocity averaging result.

\begin{theorem}
\label{VA-thm} 
Let $d\geq 2$ and assume that the sequence $(h_n)$ is uniformly compactly supported and that it converges weakly to $h_0$ in $L^q(\R^{d+1})$, for some $q> 2$.

Under assumptions {\rm a)--f)}, the sequence $(h_n)$ of solutions to \eqref{TEq} admits a (non-relabelled) subsequence such that for every $\rho \in C^1_c(\R)$ the sequence $(\int_\R \rho(\lambda) h_n(\mx,\lambda) d\lambda)$ strongly converges toward $(\int_\R \rho(\lambda) h_0(\mx,\lambda) d\lambda)$ in $L^1_{loc}(\R^{d})$.
\end{theorem}

\begin{proof}
Let $\Omega\times K \subseteq \R^{d}\times \R$ be a bounded open subset which contains supports of all the functions $h_n$.
Let us take a bounded sequence of functions $(v_n)$ uniformly compactly supported on $\Omega$ and
weakly-$\star$ converging to zero in ${L}^\infty(\Omega)$, which we take at this moment to be arbitrary. At the end 
of the proof the precise choice will be made. 
Let us pass to a subsequence of both $(h_n)$ and $(v_n)$ (not relabelled) which define the H-measure
$\mu\in L^2(K;{\cal M}(\Omega\times \Sdmj))$
in the sense of Theorem \ref{bilinearboundedness}. Depending on the regularity of $(h_n)$, the functional might be of better regularity (i.e.~in $L^q(K;{\cal M}(\Omega\times \Sdmj))$ for some $r>2$), but it is not of essential importance for the rest of the proof.
The idea of the proof is to show that $\mu=0$, which (for the right choice of $(v_n)$)
implies the desired strong convergence. 

%For a clearer presentation, in this proof we introduce the notation $\meta:=(\xi_0,\mxi)$,
%i.e.~we have $\meta=(\xi_0,\xi_1,\dots,\xi_d)$.

For arbitrary $\varphi \in {\rm C}_c(\Omega)$, $\rho\in C^1_c(K)$ and 
$\psi\in{\rm C}^{d}_c(\Sdmj)$ we set
\begin{equation*}%\label{tf-1}
\theta_n(\mx,\lambda) := \varphi(\mx)\rho(\lambda)\overline{{\cal A}_{\frac{\psi\left( \mxi/|\mxi|\right)}{|\mxi|^2}}( v_n)(\mx)} \;.
\end{equation*}
Testing \eqref{TEq} by $-\theta_n$, i.e.~multiplying it by $-\theta_n$, integrating over $\Omega\times K$, and applying the integration by parts, we get the following expression on the right-hand side:
\begin{align*}
\int\limits_{\Omega\times K} & \varphi(\mx) \rho(\lambda)\sum\limits_{k,j=1}^d  a_{kj}(\mx,\lambda) h_n(\mx,\lambda) \,\overline{{\cal A}_{\frac{4 \pi^2 \xi_j \xi_k}{|\mxi|^2}\psi(\frac{\mxi}{|\mxi|})}
	( v_n)(\mx)} \,d\mx d\lambda
\\&- 2\int\limits_{\Omega\times K}  \rho(\lambda)\sum\limits_{k,j=1}^d  a_{kj}(t,\mx,\lambda) h_n(t,\mx,\lambda) \, \pa_{x_j}\varphi(t,\mx) \overline{{\cal A}_{\frac{2 \pi i  \xi_k}{|\mxi|^2}\psi(\frac{\mxi}{|\mxi|})}
	( v_n)(t,\mx)} \,  d\mx d\lambda
\\&- \int\limits_{\Omega\times K}  \rho(\lambda)\sum\limits_{k,j=1}^d  a_{kj}(t,\mx,\lambda) h_n(t,\mx,\lambda) \, \pa_{x_j x_k}\varphi(t,\mx) \overline{{\cal A}_{\frac{1}{|\mxi|^2}\psi(\frac{\mxi}{|\mxi|})}
	( v_n)(t,\mx)} \, d\mx d\lambda \\
& +\int\limits_{\Omega\times K} G_n(t,\mx,\lambda)\varphi(t,\mx) \rho'(\lambda)\overline{
	{\cal A}_{\frac{\psi\left(\mxi/|\mxi|\right)}{|\mxi|^2}}(v_n)(t,\mx)} \, d\mx d\lambda \\
& +\sum\limits_{j=1}^d \,\int\limits_{\Omega\times K}  P^n_j(\mx,\lambda) \varphi(t,\mx) \rho(\lambda) \overline{{\cal A}_{\frac{2 \pi i\xi_j}{|\mxi|^2}\psi(\frac{\mxi}{|\mxi|})}(\varphi v_n)(\mx)} \,d\mx d\lambda\\
& +\sum\limits_{j=1}^d \,\int\limits_{\Omega\times K}  P^n_j(\mx,\lambda) \pa_{x_j}\varphi(t,\mx) \rho(\lambda) \overline{{\cal A}_{\frac{1}{|\mxi|^2}\psi(\frac{\mxi}{|\mxi|})}(\varphi v_n)(\mx)} \,d\mx d\lambda \;,
\end{align*}
where we have used $\partial_{x_j}{\cal A}_\psi={\cal A}_{2\pi i\xi_j\psi}$, which holds according to the definition of the Fourier transform. Further on, the left-hand side reads:
\begin{align*}
& \int\limits_{\Omega\times K}\sum\limits_{k=1}^d \Bigl(f_k(\mx,\lambda)+\sum\limits_{j=1}^d a_{kj,x_j}(\mx,\lambda) \Bigr)\, (\rho \, \varphi \, h_n)(\mx,\lambda)
\,\overline{{\cal A}_{\frac{2 \pi i\xi_j}{|\mxi|^2}\psi(\frac{\mxi}{|\mxi|})}
	(v_n)(\mx)}\,d\mx d\lambda \\
&\qquad + \,\int\limits_{\Omega\times K}
\rho(\lambda)\overline{{\cal A}_{\frac{1}{|\mxi|^2}\psi(\frac{\mxi}{|\mxi|})}
	( v_n)(\mx)} \, h_n(\mx,\lambda)  
\\& \qquad\qquad\qquad\qquad\qquad\qquad \times
\biggr( \sum\limits_{k=1}^d \Bigl(f_k(\mx,\lambda)+\sum\limits_{j=1}^d a_{kj,x_j}(\mx,\lambda) \Bigr)  \pa_{x_j} \varphi
\biggr) \, d\mx d\lambda \,.
\end{align*} 
Since for any $s\in (0,d)$ and $r\in (1,\frac{d}{s})$ the Riesz transform ${\cal A}_{\frac{1}{|\mxi|^s}}$ maps $L^r(\R^{d})$ continuously on $W^{s,r}_{loc}(\R^{d})$ (cf.~\cite[Chapter V]{stein}), i.e.~it is a compact mapping from $L^r(\R^{d})$ into $L_{loc}^r(\R^{d})$, and $(v_n)$ converges weakly to zero, we see that all the terms but the first one on the right-hand side tend to zero as $n\to \infty$. As for the remaining term, we have according to Theorem \ref{bilinearboundedness}:
\begin{equation*}
%\label{1st}
\begin{aligned}
\lim\limits_{n\to \infty}\int\limits_{\Omega\times K} & \varphi(\mx) \rho(\lambda)\sum\limits_{k,j=1}^d  a_{kj}(\mx,\lambda) h_n(\mx,\lambda) \,\overline{{\cal A}_{\frac{4 \pi^2 \xi_j \xi_k}{|\mxi|^2}\psi(\frac{\mxi}{|\mxi|})}
	(v_n)(\mx)} \, d\mx d\lambda\\
&= 4\pi^2\int_{\Omega\times K \times \Sdmj} \varphi(\mx) \rho(\lambda) \psi(\mxi) \langle a(\mx,\lambda)\mxi\,|\,\mxi \rangle\, d\mu(\mx,\lambda,\mxi)=0 \,.
\end{aligned}
\end{equation*} 
In other words, we conclude that 
$$
\operatorname{supp}\mu \subset A_p:=\Bigl\{(\mx,\lambda,\mxi) \in \R^{d}\times \R\times \Sdj: \,  \langle a(\mx,\lambda)\mxi\,|\,\mxi \rangle=0\Bigr\}\,,
$$ i.e.~if we denote by $\chi_{A_p^C}$ the characteristic function of the complement of the set $A_p$ we have
\begin{equation}
\label{parabolic-support}
\int_{\Omega\times K \times \Sdmj} \phi(\mx,\lambda,\mxi)  \chi_{A_p^C}(\mx,\lambda,\mxi) \,d\mu(\mx,\lambda,\mxi)=0 \,,
\end{equation}
for any $\phi \in C_0(\Omega\times K \times \Sdj)$.

To proceed, we use the property f) given above. To this end, we note that we can rewrite \eqref{TEq} in the form
\begin{equation}
\label{TEq-1order}
\begin{split}
\sum\limits_{k=1}^d \pa_{x_k}\bigl( f_k(\mx,\lambda) h_n \bigr) = \,\sum\limits_{k,j=1}^d & \pa_{x_j} \Bigl(\sigma_{kj}(\mx,\lambda) \beta^n_k(\mx,\lambda)\Bigr) \\
&+ \pa_{\lambda} G_n(\mx,\lambda) + \Div_\mx P_n(\mx,\lambda) \quad \hbox{in} \ {\cal D}'(\R^{d+1})\;.
\end{split} 
\end{equation} Similarly as in the first part of the proof and with the same notations, we take the test function
\begin{equation*}%\label{tf-2}
\tilde{\theta}_n(\mx,\lambda) := \varphi(\mx)\rho(\lambda)\overline{{\cal A}_{\frac{\psi\left( \mxi/|\mxi|\right)}{|\mxi|}}(v_n)(\mx)} \;
\end{equation*} 
and insert it into \eqref{TEq-1order} (multiplied by $-1$). On the right-hand side we get
\begin{align*}
\int\limits_{\Omega} \varphi(\mx) & \int_K  \rho(\lambda) \sum\limits_{k,j=1}^d \overline{{\cal A}_{\frac{2 \pi i\xi_j}{|\mxi|}\psi(\frac{\mxi}{|\mxi|})}
	(v_n)(\mx)} \,  \sigma_{kj}(\mx,\lambda)    d\beta^n_k \, d\mx\\
&+ \int\limits_{\Omega}   \overline{{\cal A}_{\frac{1}{|\mxi|}\psi(\frac{\mxi}{|\mxi|})}(v_n)(\mx) }\, 
	 \int_K  \rho(\lambda) \sum\limits_{k,j=1}^d \pa_{x_j}\varphi(\mx)  \,  \sigma_{kj}(\mx,\lambda)    d\beta^n_k \,  d\mx\\
& +\int\limits_{\Omega\times K} G_n(\mx,\lambda)\varphi(\mx) \rho'(\lambda)\overline{{\cal A}_{\frac{\psi\left(\mxi/|\mxi|\right)}{|\mxi|}}(v_n)(t,\mx)} \,d\mx d\lambda \\
& +\sum\limits_{j=1}^d \,\int\limits_{\Omega\times K}  P^n_j(\mx,\lambda) \varphi(t,\mx) \rho(\lambda) \overline{{\cal A}_{\frac{2 \pi i\xi_j}{|\mxi|}\psi\left(\mxi/|\mxi|\right)}(\varphi v_n)(\mx)} \,d\mx d\lambda\\
& +\sum\limits_{j=1}^d \,\int\limits_{\Omega\times K}  P^n_j(\mx,\lambda) \pa_{x_j}\varphi(\mx) \rho(\lambda) \overline{{\cal A}_{\frac{1}{|\mxi|}\psi\left(\mxi/|\mxi|\right)}(\varphi v_n)(\mx)} \,d\mx d\lambda \;,
\end{align*}
while the left-hand side reads:
\begin{align*}
\int\limits_{\Omega\times K}\sum\limits_{k=1}^d & (\rho \,\varphi\, f_k)(\mx,\lambda)  
	\overline{{\cal A}_{\frac{2 \pi i\xi_k}{|\mxi|}\psi(\frac{\mxi}{|\mxi|})}
	(v_n)(\mx)}
 \, d\mx d\lambda	 \\
&  + \,\int\limits_{\Omega\times K}
 \rho(\lambda)\overline{{\cal A}_{\frac{1}{|\mxi|}\psi(\frac{\mxi}{|\mxi|})}
	( v_n)(\mx)} \, h_n(\mx,\lambda) 
\Big( 	\sum\limits_{k=1}^d f_k(\mx,\lambda)  \pa_{x_k} \varphi
\Big) \, d\mx d\lambda 	\,. \\
\end{align*} 
Let us denote by $\mu_j$ the H-measure corresponding to the (sub)sequences $(\beta^n_j)$ and $(v_n)$ in the sense of Theorem \ref{H-for-diffusion}.
With the same argument as in the first part of the proof we see that the only 
non-zero terms on the limit $n\to\infty$ are those corresponding to the first 
term on the right-hand side and the first term on the left-hand side. 
Thus, by letting $n\to \infty$ along a common subsequence (not relabelled) defining all H-measures 
$\mu_j$, $j=1,2,\dots, d$, and $\mu$ used above, we get (the expression is divided by $2\pi i$)
\begin{equation}
\label{2nd}
\begin{split}
\int\limits_{\Omega \times K \times \Sdj} \varphi(\mx) & \rho(\lambda) \psi(\mxi) \sum\limits_{k,j=1}^d  {\xi_j} \,  \sigma_{kj}(\mx,\lambda)    d\mu_k\\
&=\int\limits_{\Omega \times K \times \Sdj}\varphi(\mx) \rho(\lambda) \psi(\mxi) \sum\limits_{k=1}^d f_k(\mx,\lambda)\xi_k d\mu \,.
\end{split}
\end{equation} 

According to \eqref{non-negativity}, we have for any $k\in\{1,2,\dots,d\}$
$$
\sum\limits_{j=1}^d {\xi_j} \,  \sigma_{kj}(\mx,\lambda)=0 \;, \quad
	(\mx,\lambda,\mxi)\in A_p \,.
$$ 
Thus, from \eqref{2nd} for every $\phi\in C_0(\Omega\times K \times \Sdj)$ it holds:
\begin{equation*}
\begin{split}
0 &=\int\limits_{A_p} \phi(\mx,\lambda,\mxi) \sum\limits_{k,j=1}^d  {\xi_j} \,  \sigma_{kj}(\mx,\lambda)  \,  d\mu_k=\int\limits_{A_p}\phi(\mx,\lambda,\mxi) \langle f(\mx,\lambda)\,|\,\mxi\rangle \, d\mu \;.
\end{split}
\end{equation*}
Denoting
$$
A_f:=\Bigl\{(\mx,\lambda,\mxi) \in \Omega\times K \times \Sdj : \,
	\langle f(\mx,\lambda)\,|\,\mxi\rangle=0 \Bigr\}
$$ 
and by $A_f^C$ its complement, we get from the former relation
\begin{equation}
\label{2nd-1}
\int\limits_{A_p\cap A^C_f}\phi(\mx,\lambda,\mxi)\, d\mu=0 \,.
\end{equation}

Combining \eqref{parabolic-support} and \eqref{2nd-1}, we obtain
\begin{equation*}
%\label{sum}
\begin{split}
0 &=\int\limits_{\Omega\times K \times \Sdj}\phi(\mx,\lambda,\mxi) \Bigl(\chi_{A_p\cap A^C_f}(\mx,\lambda,\mxi)+\chi_{A^C_p}(\mx,\lambda,\mxi)\Bigr) \,d\mu\\
&=\int\limits_{\Omega\times K \times \Sdj}\phi(\mx,\lambda,\mxi) \Bigl(1-\chi_{A_p\cap A_f}(\mx,\lambda,\mxi)\Bigr) \,d\mu\,,
\end{split}
\end{equation*}
hence
\begin{align*}
\int\limits_{\Omega\times K \times \Sdj}\phi(\mx,\lambda,\mxi) \,d\mu 
	&= \int\limits_{\Omega\times K \times \Sdj}\phi(\mx,\lambda,\mxi) \chi_{A_p\cap A_f}(\mx,\lambda,\mxi)\,d\mu \\
&=\int\limits_{\Omega \times \Sdj}\int_K\phi(\mx,\lambda,\mxi) \chi_{A_p\cap A_f}(\mx,\lambda,\mxi)g(\mx,\lambda,\mxi)\, d\lambda d\nu(\mx,\mxi)\,,
\end{align*}
where we have used the representation \eqref{repr}. In order to obtain that $\mu$ is a zero measure, 
it is sufficient to show that the right-hand side in the expression above equals zero. 
This holds due to the non-degeneracy condition \eqref{lp-111}.
Indeed, we have that for a.e.~$\mx\in\Omega$ with respect to the Lebesgue measure and for any $\mxi\in\Sdj$
$$
 \chi_{A_p\cap A_f}(\mx,\lambda,\mxi)=0 
$$ 
for a.e.~$\lambda\in K$ with respect to the Lebesgue measure. Thus, the inner integral is equal to zero 
for a.e.~$\mx\in\Omega$ with respect to the Lebesgue measure and for any $\mxi\in\Sdj$.
It is left to note that by Theorem \ref{mu-repr-h} the $\Sdmj$-projection of the measure $\nu$ is absolutely continuous
with respect to the Lebesgue measure, hence the inner integral is equal to zero for 
$\nu$-a.e.~$(\mx,\mxi)\in\Omega\times\Sdj$.
Therefore, $\mu=0$.

Let us finally discuss on the choice of $(v_n)$. 
To this end, let
$$
V_n(\mx)=\int_\R \Bigl(h_n(\mx,\lambda)-h_0(\mx,\lambda)\Bigr)\rho(\lambda) d\lambda\,,
$$
and let us define
$$
\tilde v_n(\mx)=
\left\{
\begin{array}{lcr}
\frac{V_n(\mx)}{|V_n(\mx)|} &,& V_n(\mx)\neq 0 \,,\\
0 &,& V_n(\mx)=0 \,.
\end{array}
\right.
$$
Since $(\tilde v_n)$ is bounded in $L^\infty(\Omega)$, it has a weakly-$*$
converging subsequence, whose limit we denote by $\tilde v\in L^\infty(\Omega)$.
Passing to that subsequence (not relabelled), we define
$$
v_n := \tilde v_n -\tilde v \,,
$$
which obviously has all the properties used in the proof for $(v_n)$.
Now we have
\begin{align*}
\lim\limits_{n\to \infty} & \int_\Omega \varphi(\mx) \biggl| \int_K \rho(\lambda) \Bigl(h_n(\mx,\lambda)-h_0(\mx,\lambda)\Bigr) d\lambda\biggr|\,  d\mx \\
&= \int_{\Omega\times K}\varphi(\mx)\rho(\lambda)\Bigl(h_n(\mx,\lambda)-h_0(\mx,\lambda)\Bigr) 
	\overline{\tilde v_n(\mx)} \, d\mx d\lambda \\
&= \int_{\Omega\times K}\varphi(\mx)\rho(\lambda)h_n(\mx,\lambda) 
\overline{v_n(\mx)} \,d\mx d\lambda \\
&= \mu(\varphi\otimes\rho\otimes 1) = 0 \,.
\end{align*}
Thus, the proof is over.
\end{proof}

\begin{remark}
The result of the previous theorem would remain true if we ask for the non-degeneracy condition given 
in part c) of Introduction to hold only on fixed compact sets $\Omega, K$ containing the support of all
$h_n$'s.
\end{remark}

\smallskip

\section{Fully heterogeneous degenerate parabolic equation with rough coefficients -- existence proof}\label{sec:application}
% =========================================================================

In this section, we prove existence of a weak solution to the Cauchy problem corresponding to equation \eqref{d-p} augmented with the initial conditions

\begin{align}
%\label{d-p} \pa_t u + \Div_\mx \bigl(\mff(t,\mx, u(t,\mx)) \bigr) =& \,\Div_\mx  \left( a(t,\mx, u(t,\mx)) \nabla u \right) + s(t,\mx,u), \\
\label{ic}
u|_{t=0}&=u_0 \in {L}^1({\R}^d)\cap {L}^\infty({\R}^d) \;.
\end{align}

In addition to a), b), and c), we assume:
\begin{itemize}
\item[i)] there exist $m,M\in\R$ such that the initial data $u_0\in{L}^1({\R}^d)\cap {L}^\infty({\R}^d)$ 
is bounded between $m$ and $M$ and the flux and diffusion equal zero at $\lambda=m$ and $\lambda=M$:
\begin{equation}
\label{bnd-fl}
m\leq u_0(\mx) \leq M \qquad {\rm and} \qquad 
\begin{aligned}
  &\mff(t,\mx,m)=\mff(t,\mx,M)=0 \\
 & a(t,\mx,m)=a(t,\mx,M)=0 
\end{aligned}
\qquad  {\rm a.e.} \ \ (t,\mx)\in \R^{d+1}_+ \;;
\end{equation}

\item[ii)] the convective term $\mff=\mff(t,\mx,\lambda)$ belongs to $L^p_{loc}(\R^{d+1}_+; C^1([m,M])^d)$ for some $p>1$, 
$$
\sup_{\lambda\in [m,M]}|f(\cdot,\cdot,\lambda)|\in L^p_{loc}(\R^{d+1}_+) \,,
$$
and
\begin{equation*}
%\label{meas} 
\Div_{\mx} \mff \in
{\cal M}(\R^{d+1}_+\times [m,M]) \,,
\end{equation*} 
where ${\cal M}(X)$ denotes the space of Radon measures on $X\subseteq\Rd$; 

\item[iii)] the matrix $A(t,\mx,\lambda)\in C^{1,1}(\R^{d+1}\times [m,M];\R^{d\times d})$
is symmetric and non-decreasing with respect to $\lambda\in [m,M]$, i.e.~the (diffusion) matrix $a(t,\mx,\lambda):=A'(t,\mx,\lambda)$ satisfies
\begin{equation*}
\langle a(t,\mx,\lambda)\mxi \,|\, \mxi \rangle \geq 0 \,,
\end{equation*} 
and $a$ satisfies b) from the Introduction.
%
%\item[iv)] $f:=\pa_\lambda\mff$ and $a=A'$ satisfy the non-degeneracy assumption given in f) from the previous section.
\end{itemize}

\begin{remark}\label{remX} 
We note that condition (i) essentially provides the maximum principle for entropy solutions of \eqref{d-p}+\eqref{ic}. More precisely, if the flux is $C^1$ with respect to all the variables and the initial data $u_0$ are bounded between two stationary (and thus both entropy and weak) solutions $m$ and $M$ (which is why we assume \eqref{bnd-fl}), then the entropy solution to \eqref{d-p}+\eqref{ic} remains bounded between $m$ and $M$. This is called the maximum principle (see e.g.~\cite{CK}). % For more details, see the proof of Theorem \ref{main-tvr}. 
\end{remark}

We are going to work with the entropy solutions in the sense given in \cite[Definition 2.1]{CK}. The conditions ({\bf D}.1) and ({\bf D}.2) describe regularity effects of the diffusion term in \eqref{d-p}. We shall rewrite them in the form \eqref{regularity}. More precisely, we have the following lemma.

\begin{lemma}
\label{condf}
Assume that the flux $\mff$ and diffusion $a$ from \eqref{d-p} are regular in the sense that 
$$
\mff \in C^1(\R^{d+1}_+\times \R;\R^d) \ \ {\rm and} \ \ a\in C^2(\R^{d+1}_+\times \R;\R^{d\times d})\,. 
$$ Then, a bounded entropy solution to \eqref{d-p} in the sense of \cite[Definition 2.1]{CK} satisfies the following equality
\begin{equation}
\label{regularity-2}
\begin{split}
\sum\limits_{k,j=1}^d \pa^2_{x_k x_j}\Bigl( & a_{kj}(t,\mx,\lambda) \sgn\bigl(u(t,\mx)-\lambda\bigr) \Bigr)
	-\sum\limits_{k,j=1}^d \pa_{x_k} \Bigl( a_{kj,x_j}(t,\mx,\lambda) \sgn\bigl(u(t,\mx)-\lambda\bigr) \Bigl)\\
&= \sum\limits_{k,j=1}^d \pa_{x_j} \Bigl(\sigma_{kj}(t,\mx,\lambda) \beta_k(t,\mx,\lambda)\Bigr)
\end{split}
\end{equation} in the sense of distributions, where
\begin{equation}
\label{beta-es}
\beta_k(t,\mx,\lambda)=2\delta\bigl(u(t,\mx)-\lambda\bigr) \sum\limits_{s=1}^d \biggl( \pa_{x_s}\Bigl(\int_0^{u(t,\mx)} \sigma_{sk}(t,\mx,w)\, dw \Bigr)- \int_0^{u(t,\mx)}\sigma_{sk,x_s}(t,\mx,w) \,dw\biggr) \,.
\end{equation}
\end{lemma}
\begin{proof}
Let us first explain the meaning of \eqref{beta-es}. It is an object belonging to $L^2_{loc}(\R^{d+1}_+;{\cal M}(\R))$ (see \eqref{beta-bnd}) defined for every $\rho\in C_0(\R)$ and $\varphi\in L^2_c(\R^{d+1}_{+})$ by
\begin{equation}
\label{beta-es-1}
\begin{split}
&\int_{\R^{d+1}_+} \varphi(t,\mx) \int_{\R} \rho(\lambda) d\beta_k(t,\mx,\lambda) \,dtd\mx
\\&=2\int_{\R^{d+1}_+} \varphi(t,\mx) \rho(u(t,\mx)) \sum\limits_{s=1}^d \biggl( \pa_{x_s}\Bigl(\int_0^{u(t,\mx)} \sigma_{sk}(t,\mx,w)\, dw \Bigr)- \int_0^{u(t,\mx)}\sigma_{sk,x_s}(t,\mx,w) \,dw\biggr) dt d\mx
\\&=2\int_{\R^{d+1}_+} \varphi(t,\mx) \sum\limits_{s=1}^d \biggl( \pa_{x_s}\Bigl(\int_0^{u(t,\mx)} \rho(w)\sigma_{sk}(t,\mx,w)\, dw \Bigr)- \int_0^{u(t,\mx)}\rho(w)\sigma_{sk,x_s}(t,\mx,w) \,dw\biggr) dt d\mx \,,
\end{split}
\end{equation}
where in the last step we have used the chain rule given in \cite[({\bf D}.2)]{CK}. We also note that the computations from the above are correct since by \cite[({\bf D}.1)]{CK}
$$
\sum\limits_{s=1}^d \biggl( \pa_{x_s}\Bigl(\int_0^{u(t,\mx)} \sigma_{sk}(t,\mx,w)\, dw \Bigr)- \int_0^{u(t,\mx)}\sigma_{sk,x_s}(t,\mx,w) \,dw\biggr) \in L_{loc}^2(\R^{d+1}_+)\,.
$$ From the same reason, it also holds 
\begin{equation}
\label{beta-bound}
\| \beta_k(t,\mx,\cdot) \|_{{\cal M}(\R)}=\sum\limits_{s=1}^d \biggl( \pa_{x_s}\Bigl(\int_0^{u(t,\mx)} \sigma_{sk}(t,\mx,w)\, dw \Bigr)- \int_0^{u(t,\mx)}\sigma_{sk,x_s}(t,\mx,w) \,dw\biggr) \in L_{loc}^2(\R^{d+1}_+).
\end{equation}

Next, note that the entropy solution to \eqref{d-p}+\eqref{ic}, is obtained as the strong $L^1_{loc}(\R^{d+1}_+)$ limit of the vanishing viscosity approximation to \eqref{d-p} (see \cite[Section 6]{CP}). Denote the family of solutions obtained by the vanishing viscosity approximation by $(u_\eps)$ and remark that, by taking a smooth approximation of the initial data, we can assume that $u_\eps$ are smooth functions. It satisfies (by the direct application of the derivative product rule and using $a_{jk}=a_{kj}=\sum\limits_{s=1}^d\sigma_{ks} \, \sigma_{sj}$)
\begin{equation}
\label{regularity-3}
\begin{split}
\sum\limits_{k,j=1}^d \pa^2_{x_k x_j}\Bigl( & a_{kj}(t,\mx,\lambda) \sgn\bigl(u_\eps(t,\mx)-\lambda\bigr) \Bigr)
	-\sum\limits_{k,j=1}^d \pa_{x_k} \Bigl( a_{kj,x_j}(t,\mx,\lambda) \sgn\bigl(u_\eps(t,\mx\bigr)-\lambda) \Bigr)\\
&= \sum\limits_{k,j=1}^d \pa_{x_j} \Bigl(\sigma_{kj}(t,\mx,\lambda) \beta^\eps_k(t,\mx,\lambda)\Bigr)\,,
\end{split}
\end{equation}
where 
\begin{equation*}
%\label{beta-eps}
\begin{split}
\beta^{\eps}_k(t,\mx,\lambda)&=2\delta\bigl(u_\eps(t,\mx)-\lambda\bigr) \sum\limits_{s=1}^d \biggl( \pa_{x_s}\Bigl(\int_0^{u_\eps(t,\mx)} \sigma_{sk}(t,\mx,w)\, dw \Bigr)- \int_0^{u_\eps(t,\mx)}\sigma_{sk,x_s}(t,\mx,w) \,dw\biggr)
\\&=2\delta\bigl(u_\eps(t,\mx)-\lambda\bigr) \sum\limits_{s=1}^d  \sigma_{sk}\bigl(t,\mx,u_\eps(t,\mx)\bigr) \pa_{x_s} u_\eps(t,\mx)\\
&= \sum\limits_{s=1}^d  \sigma_{sk}(t,\mx,\lambda) 
	\pa_{x_s} \sgn\bigl(u_\eps(t,\mx)-\lambda\bigr) \,,
\end{split}
\end{equation*} 
and the last representation of $\beta^{\eps}_k$ makes \eqref{regularity-3} clear.

Now, we note that by letting $\eps\to 0$, the left-hand side of \eqref{regularity-3} converges toward the left-hand side of \eqref{regularity-2} in the sense of distributions (since $u_\eps \to u$ strongly in $L^1_{loc}(\R^{d+1}_+)$). As for the right-hand side, we note that $u_\eps$ satisfies the relation analogical to \eqref{beta-es-1}: 
\begin{equation*}
%\label{beta-es-\eps}
\begin{split}
&\int_{\R^{d+1}_+} \varphi(t,\mx) \int_{\R} \rho(\lambda) d\beta^\eps_k(t,\mx,\lambda) \,dtd\mx\\
&=2\int_{\R^{d+1}_+} \varphi(t,\mx) \sum\limits_{s=1}^d \biggl( \pa_{x_s}\Bigl(\int_0^{u_\eps(t,\mx)}\rho(w) \sigma_{sk}(t,\mx,w)\, dw \Bigr)\\
& \qquad\qquad\qquad\qquad\qquad- \int_0^{u_\eps(t,\mx)}\rho(w)\sigma_{sk,x_s}(t,\mx,w) \,dw\biggr) dt d\mx\\
&=-2 \sum\limits_{s=1}^d \int_{\R^{d+1}_+}  \biggl((\pa_{x_s}\varphi)(t,\mx)\int_0^{u_\eps(t,\mx)}\rho(w) \sigma_{sk}(t,\mx,w) \,dw \\
& \qquad\qquad\qquad\qquad\qquad 
+\varphi(t,\mx)\int_0^{u_\eps(t,\mx)}\rho(w)\sigma_{sk,x_s}(t,\mx,w) \,dw
\biggr) dt d\mx \,,
\end{split}
\end{equation*}
which also converges toward the expression for $\beta_k$ given by \eqref{beta-es-1}. This concludes the proof.
\end{proof}

The main theorem of the section is the following.

\begin{theorem}
\label{main-tvr} Let $d\geq 1$, and let $u_0,\mff$ and $A$ satisfy conditions {\rm a)-c)} and {\rm i)-iv)} above. 
Then there exists a weak solution to \eqref{d-p} augmented with the initial condition \eqref{ic}.
\end{theorem}

\begin{proof} 
Consider the sequence of admissible solutions to the following regularised Cauchy problems
\begin{align}
\label{d-p-r}
&\pa_t u_n+\Div_{\mx} \mff_n(t,\mx,u_n)=\Div_\mx  \left( a(t,\mx, u_n) \nabla u_n \right)  \,,  
\\
&u_n|_{t=0}=u_0 \in {L}^1(\R^d)\cap {L}^\infty(\R^d) \,,
\label{i-c-r}
\end{align} 
where $\mff_n$ is a smooth regularisation with respect to $(t,\mx)$ of $\mff$
obtained by the convolution of $\mff$ with a standard mollifier.
Thus,
\begin{equation}\label{eq:fn-alpha-beta}
\mff_n(t,\mx,m)=\mff_n(t,\mx,M)=0 \;,
\end{equation}
and for any compact $K\subseteq\R^{d+1}_+$ it holds (see \cite[(4.10)]{LM5})
\begin{equation*}
%\label{f-conv}
\lim_{n\to\infty}\Bigl\| \sup_{\lambda\in[m,M]}\bigl|\mff_n(\cdot,\cdot,\lambda)-\mff(\cdot,\cdot,\lambda)\bigr|\Bigr\|_{{L}^p(K)} =
\lim_{n\to\infty}\Bigl\| \sup_{\lambda\in[m,M]}\bigl|\partial_\lambda\mff_n(\cdot,\cdot,\lambda)-\partial_\lambda\mff(\cdot,\cdot,\lambda)\bigr|\Bigr\|_{{L}^p(K)} =0 \,.
\end{equation*}

We denote by $A$ the matrix valued function such that $\pa_\lambda A(t,\mx,\lambda)=a(t,\mx,\lambda)$.

It is well known that there exists a solution $u_n$ to such an equation
satisfying for a.e.~$\lambda\in \R$ the Kru\v zkov-type entropy equality
(see \cite{CK} where existence was shown under more restrictive conditions), i.e.
\begin{align}
\label{e-c-n} \pa_t |u_n-\lambda|&+\Div_\mx\Bigl( \sgn(u_n-\lambda)\bigl(\mff_n(t,\mx,u)-\mff_n(t,\mx,\lambda) \bigr)
\Bigr)\\ \nonumber
&-\sum\limits_{k,j=1}^d \pa_{x_j x_k} \Bigl( \sgn(u_n-\lambda)
\bigl(A_{kj}(t,\mx,u_n)-A_{kj}(t,\mx,\lambda)\bigr)\Bigr)\\
&+ \sum\limits_{k,j=1}^d \pa_{x_j} \Bigl( \sgn(u_n-\lambda) \left( A_{kj,x_k}(t,\mx,u_n)-A_{kj,x_k}(t,\mx,\lambda)  \right) \Bigr) = \zeta_n(t,\mx,\lambda) \,,
\nonumber
\end{align} 
where $\zeta_n\in {\rm C}(\R_\lambda;w\star-{\cal M}(\R^{d+1}_+))$ and ${\cal M}(\R^{d+1}_+)$ is the space of Radon measures. We also note that the sequence $(\zeta_n)$ is bounded in ${\cal M}(\R^{d+1}_+\times \R)$ (when we consider $\zeta_n$ as a Radon measure in all three variables).

By finding derivative with respect to $\lambda$ in \eqref{e-c-n}, we see that the (kinetic) function
\begin{equation*}
%\label{equil-n} 
h_n(t,\mx,\lambda)=\sgn(u_n(t,\mx)-\lambda)=
 -\pa_\lambda |u_n(t,\mx)-\lambda|
\end{equation*} is a weak solution to the following linear equation:
\begin{align}
\label{kinetic-n} 
\pa_t h_n &+ \Div_\mx \bigl( {f}(t,\mx,\lambda)
h_n\bigr)- 
\sum\limits_{k,j=1}^d \pa_{x_j x_k} \Bigl( h_n \, a_{kj}(t,\mx,\lambda)\Bigr)\\
&+ \sum\limits_{k,j=1}^d \pa_{x_j} \Bigl( a_{kj,x_k}(t,\mx,\lambda) \, h_n  \Bigr)=\Div_\mx \bigl( (f-{f}_n)(t,\mx,\lambda)
h_n\bigr) -\pa_\lambda
\zeta_n(t,\mx,\lambda)\, , \nonumber
\end{align} where ${f}_n=\pa_\lambda\mff_n$, ${f}=\pa_\lambda\mff$, and $a=A'$. 

We aim to show that $(h_n)$ satisfies conditions of Theorem \ref{VA-thm}. To this end, we note that by Lemma \ref{condf}, condition f) of Theorem \ref{VA-thm} is fulfilled.

According to i) and \eqref{eq:fn-alpha-beta}, we know that $(u_n)$ satisfies the maximum principle, i.e.~every $u_n$, $n\in \N$, remains bounded between $m$ and $M$ (see Remark \ref{remX}). Therefore, the sequence $(\zeta_n)$ is bounded 
in $\mathrm{C}(\R_\lambda; w\star-{\cal M}(\R^{d+1}_+))$. This implies that $(\zeta_n)$ is actually strongly precompact in ${L}^r_{loc}(\R_\lambda; \mathrm{W}_{loc}^{-1,q}(\R^{d+1}_+))$ 
for any $r\geq 1$ and $q\in \bigl[1,\frac{d+1}{d}\bigr)$ 
(one can prove this in the same manner as in \cite[Theorem 1.6]{Evans} using 
that $\mathrm{W}^{\frac{1}{2},s}(\R^{d+1})$ is compactly embedded into 
$\mathrm{C}(\mathop{Cl} K)$, $K\subset\subset\R^{d+1}$, for $\frac{s}{2}<d+1$), 
and denote by $\zeta$ the limit along an appropriate subsequence 
(not relabelled and implied in what follows). This provides conditions d) and e).

From the above, we see that equation \eqref{kinetic-n} satisfies conditions of Theorem \ref{VA-thm}. More precisely, if we denote by $h$ the weak-$\star$ limit along a subsequence of the sequence $h_n$, we have from Theorem \ref{VA-thm} that  (along the non-relabeled subsequence)
$$
\int_{\R} \rho(\lambda) (h_n(t,\mx,\lambda)-h(t,\mx,\lambda)) d\lambda \to 0 \ \ {\rm as} \ \ n\to \infty
$$ strongly in ${L}^1_{loc}(\R^{d+1}_+)$ for any $\rho\in{\rm C}^1_c(\R)$.
Due to density arguments, we can insert $\rho(\lambda)=\chi_{[m,M]}$, obtaining
\begin{align*}
2u_n(t,\mx)-m-M &= \int_{m}^{u_n(t,\mx)} d\lambda - \int_{u_n(t,\mx)}^{M} d\lambda \nonumber \\
&= \int_{m}^{M} {\rm sgn}(u_n(t,\mx)-\lambda ) \,d\lambda
\overset{n\to\infty}{\longrightarrow} \int_{m}^{M} h(t,\mx,\lambda) \,d\lambda \,,% \label{eq:strong_conv_un}
\end{align*} where the latter convergence is, as before, in ${L}^1_{loc}(\R^{d+1}_+)$.
Therefore, $(u_n)$ strongly converges in ${L}^1_{loc}(\R^{d+1}_+)$ toward
$$
u(t,\mx) := {\int_{m}^{M} h(t,\mx,\lambda) d\lambda\,+{m}+{M}\over 2} \,.
$$

Clearly, the latter strong convergence implies that the function $u$ is a weak solution to \eqref{d-p}+\eqref{ic}. \end{proof}

\section{Regularity of entropy solutions to homogeneous degenerate parabolic equations}

In this section, we are interested in the Sobolev regularity properties of entropy solutions to the homogeneous (autonomous) variant of \eqref{d-p}:

\begin{equation}\label{d-p-hom}
\begin{split}
\pa_t u + \Div_\mx \bigl(\mff(u(t,\mx)) \bigr) =& \,\Div_\mx  \bigl( a(u(t,\mx)) \nabla_\mx u \bigr) \qquad \hbox{in} \ {\cal D}'(\R_+^{d+1})\;,
\end{split}
\end{equation} where $(t,\mx)\in \R_+^{d+1}=[0,\infty)\times \R^d$, $u\in L_{loc}^2(\R_+^{d+1})\cap L^1(\R_+^{d+1})$ is an unknown function, $\mff \in C^1(\R;\R^d)$ and the diffusion matrix function
$a=(a_{ij})_{i,j=1,\dots,d}\in C^{0,1}(\R;\R^{d\times d})$ satisfies b) form Introduction.

Definition of an entropy solution to \eqref{d-p-hom} is introduced in \cite[Definition 2.2]{CP}. We have already used its generalisation from \cite{CK} in the previous section and we shall keep the same notation bearing in mind that the coefficients are $(t,\mx)$--independent.

We note that in the homogeneous situation, the objects $\beta_k$ from \eqref{beta-es} become
\begin{equation}
\label{beta-es-hom}
\beta_k(t,\mx,\lambda)=2\delta(u(t,\mx)-\lambda) \sum\limits_{s=1}^d 
\pa_{x_s} \int_0^{u(t,\mx)}\sigma_{sk}(w)\,dw \;, \quad k=1,2,\dots,d\,.
\end{equation}
In other words, for any $\rho\in C_0(\R)$ and $\varphi\in L^2_c(\R^{d+1}_{+})$, we have (see also \eqref{beta-es-1})
\begin{equation}\label{eq:beta_interpretation}
\begin{aligned}
\int_{\R^{d+1}_+} & \varphi(t,\mx) \int_{\R} \rho(\lambda) d\beta_k(t,\mx,\lambda) \,dtd\mx\\
&=2\int_{\R^{d+1}_+} \varphi(t,\mx) \rho(u(t,\mx)) \sum\limits_{s=1}^d  \pa_{x_s}\int_0^{u(t,\mx)}\sigma_{sk}(w)\,dw \,dt d\mx \,.
\end{aligned}
\end{equation} 
Furthermore, in the same way as in \eqref{beta-bound}, we have $\beta_k\in L^2_{loc}(\R^{d+1}_+;{\cal M}(\R))$ in the sense of \eqref{beta-bnd}. The following lemma is a direct corollary of Lemma \ref{condf}. It is obtained by simply removing the $(t,\mx)$--dependency of the coefficients in Lemma \ref{condf}.
%we have
%\begin{align*} 
%\Bigl| \int_{\R^{d+1}_+} & \varphi(t,\mx) \int_{\R} \rho(\lambda) d\beta_k(t,\mx,\lambda) \,dtd\mx \Bigr| \\
%&\lesssim \bigl\| \|\varphi(\cdot_t,\cdot_\mx)\rho(\cdot_\lambda)\|_{C_0(\R_\lambda)} \bigr\|_{L^2(\R^{d+1}_+)}=\|\varphi(\cdot_t,\cdot_\mx) \|_{L^2(\R^{d+1}_+)} \|\rho(\cdot_\lambda)\|_{C_0(\R_\lambda)} \,.
%\end{align*}

\begin{lemma}
	%\label{condf-hom}
	A bounded entropy solution $u$ to \eqref{d-p-hom} (in the sense of \cite[Definition 2.2]{CP}) satisfies the following equality
	\begin{equation*}
	%\label{regularity-2-hom}
	\begin{split}
	&\sum\limits_{k,j=1}^d \pa_{x_k x_j}\big( a_{kj}(\lambda)\, {\rm sgn}(u(t,\mx)-\lambda) \big)= \sum\limits_{k,j=1}^d \pa_{x_j} \bigl(\sigma_{kj}(\lambda) \beta_k(t,\mx,\lambda)\bigr)
	\end{split}
	\end{equation*} 
	in the sense of distributions, where ${\rm sgn}$ denotes the signum function and $\beta_k$ are given by \eqref{beta-es-hom}.
\end{lemma}

%
%We are going to prove a regularity result for entropy solutions to \eqref{d-p-hom} under globally defined non-degeneracy conditions given by
%
%\begin{equation}
%\label{nondeg-1}
%\sup\limits_{\mxi \in \R^d} 
%	\operatorname{meas}\bigl\{\lambda \in K :\, \big|\langle f(\lambda)\,|\,\frac{\mxi}{|\mxi|}
%	\rangle \big|^2+|\mxi|^{1-r}\langle a(\lambda)\frac{\mxi}{|\mxi|} \,|\,\frac{\mxi}{|\mxi|} \rangle \leq \eps \bigr\}\leq \eps^{\tilde{\alpha}(r)}
%\end{equation} for some small enough $r>0$ and some $\tilde{\alpha}(r)>0$. 

At this moment, to avoid proliferation of symbols, we shall incorporate the variable $t$ into the space variables. Moreover, we shall introduce  a simplified variant of the definition of entropy solutions (in accordance to the one from \cite{CP}) since it is enough for the proof of the solution regularity.

\begin{definition}
\label{def-entropy}
We say that the function $u\in L^2(\R^d)\cap L^1(\R^d)$ is an entropy solution to 
\begin{equation}
\label{d-p-cauchy}
\begin{split}
\Div_\mx \bigl(\mff( u(\mx)) \bigr) =& \,\Div_\mx  \left( a(u(\mx)) \nabla u \right)
%u|_{t=0}=&u_0\in L^2(\R^d)
\end{split}
\end{equation} if for $h(\mx,\lambda)={\rm sgn}(u(\mx)-\lambda)$ the following kinetic equation is satisfied (recall that $f=\mff'$)
\begin{equation}
\label{q-sol-reg-1}
\begin{split}
&\Div_\mx \bigl( {f}(\lambda)
h\bigr)- 
\sum\limits_{k,j=1}^d \pa_{x_j x_k} \bigl( a_{kj}(\lambda) \, h \bigr)= -\pa_\lambda
\zeta(\mx,\lambda)
\end{split}
\end{equation} 
for a non-negative measure $\zeta\in {\cal M}_{loc}(\R^d\times \R) \cap C(\R_\lambda;{\cal M}_{loc}(\R_{\mx}^d))$ together with the following regularity condition %for every $\rho \in C^1_c(\R)$
\begin{equation}
\label{regularity-1}
\begin{split}
&\sum\limits_{k,j=1}^d \pa_{x_k x_j}\big( a_{kj}(\lambda) {\rm sgn}(u(\mx)-\lambda) \big)= \sum\limits_{k,j=1}^d \pa_{x_j} \big(\sigma_{kj}(\lambda) \beta_k(\mx,\lambda)\big)
\end{split}
\end{equation} 
in the sense of distributions, where $\beta=(\beta_1,\dots,\beta_d)$ 
and
\begin{equation}
\label{beta-es-reg}
\beta_k(\mx,\lambda):=2\delta(u(\mx)-\lambda) \sum\limits_{s=1}^d  \pa_{x_s}\Bigl(\int_0^{u(\mx)} \sigma_{sk}(w) \,dw \Bigr) \;, \quad k=1,2,\dots,d\,,
\end{equation} 
is understood in the sense of \eqref{eq:beta_interpretation} and it is bounded in $L_{loc}^2(\R^d;{\cal M}_{loc}(\R))$.
\end{definition}

Now we present the main result of the section on the regularity property for entropy solutions to \eqref{d-p-cauchy} 
(and then \eqref{d-p-hom} as well) under globally defined non-degeneracy conditions \eqref{non-deg-stand}.

\begin{theorem}
	\label{t-regularity} 
	Let $d\geq 2$, $I\subseteq \R$ be a compact interval, $\mff\in C^{1,1}(I;\Rd)$ and let $a\in C^{0,1}(I;\R^{d\times d})$ 
	satisfies {\rm b)} (on $I$) from Introduction. Furthermore, assume that 
	the non-degeneracy condition \eqref{non-deg-stand} is satisfied for some $\alpha, \delta_0>0$ 
	on $I$.	
	
	If $u$ is the entropy solution to \eqref{d-p-cauchy} in the sense of Definition \ref{def-entropy}, then for every $\rho\in C^1_c(I)$, $q \in (1,q_*)$ and $s\in (0,s_*)$ it holds $\int_I h(\mx,\lambda) \rho(\lambda) d\lambda \in W_{loc}^{s,q}(\R^d)$,
	where
	\begin{align*}
	q_*&=q_*(\alpha,d):=\frac{2\alpha+2(2d+6)}{2\alpha+2d+6} \;,\\
		s_*&=s_*(\alpha,d):=\frac{\alpha^2(\alpha+4)}{(3\alpha+8)\bigl((\alpha+4)(\alpha+2(d+1))+2\alpha\bigr)} \;.
	\end{align*}
\end{theorem}

\begin{proof}
The proof of Theorem \ref{t-regularity} is divided into six steps. 
The first one is just a preliminary one in which we use the linear structure 
of the problem to localise coefficients and the solution on a compact set. 
In the second step we employ the method 
from Section \ref{sec:velocity_averaging} by which 
a new kinetic formulation is obtained. 
In the next two steps, we derive bounds in the case when the symbol is regular (Step 3) 
and then when it degenerates, i.e.~vanishes (Step 4).
The crucial point in Step 4 is the usage of the non-degeneracy condition
(and in fact that is the only place where it is used). 
The remaining part of the proof deals with the interpolation and derivation 
of the final estimate. This last part is already well established and standard
(cf.~\cite[pp.~2516--2518]{GH}).

\smallskip

\noindent \textbf{Step I.}~Since equation \eqref{q-sol-reg-1} is linear, we can multiply it by a cutoff function $\varphi\in C^2_c(\R_\mx^d)$ and we will end up with an equation of essentially the same type, 
but with the solutions $\varphi h$ which are compactly supported with respect to $\mx$. 
Indeed, by multiplying \eqref{q-sol-reg-1} by $\varphi \in C^2_c(\R^d)$ and rearranging, we get
\begin{equation*}
\begin{split}
\Div_\mx \bigl(f(\lambda) \,& \varphi(\mx) \, h(\mx,\lambda)\bigr)-\sum\limits_{k,j=1}^d \pa_{x_j x_k} \Bigl( a_{kj}(\lambda) \, \varphi(\mx) \, h(\mx,\lambda)\Bigr)\\
&= \langle\nabla \varphi(\mx) \,|\, f(\lambda)\rangle \, h(\mx,\lambda)-2\sum\limits_{k,j=1}^d \pa_{x_j} \Bigl( \pa_{x_k}\varphi(\mx) a_{kj}(\lambda) \,  \, h(\mx,\lambda)\Bigr)\\
& \quad\quad +\sum\limits_{k,j=1}^d  \Bigl( \pa_{x_k x_j}\varphi(\mx) a_{kj}(\lambda) \,  \, h(\mx,\lambda)\Bigr)
-\pa_\lambda \big( \varphi(\mx) \zeta(\mx,\lambda) \big) \,,
\end{split}
\end{equation*} and the terms on the right-hand side do not affect the proof since they are of lower order than the corresponding terms on the left-hand side (see derivation of \eqref{axi}). Similarly, we rewrite \eqref{q-sol-reg-1} using \eqref{regularity-1} 
\begin{equation}
\label{q-sol-reg-2}
\begin{split}
&\Div \bigl( {f}(\lambda)
h\bigr)- 
\sum\limits_{k,j=1}^d \pa_{x_j} \bigl( \sigma_{kj}(\lambda) \, \beta_k \bigr)= -\pa_\lambda
\zeta(\mx,\lambda) \,
\end{split}
\end{equation} 
and multiply it by $\varphi$ to obtain the equation of (almost) the same type:
\begin{equation*}
\begin{split}
\Div \bigl( {f}(\lambda) &
\varphi(\mx) h(\mx,\lambda) \bigr)- 
\sum\limits_{k,j=1}^d \pa_{x_j} \Bigl( \sigma_{kj}(\lambda) \, \varphi(\mx)  \beta_k \Bigr)\\
&= \langle\nabla \varphi(\mx) \,|\, {f}(\lambda)\rangle\,
 h(\mx,\lambda) - 
\sum\limits_{k,j=1}^d   \sigma_{kj}(\lambda) \, \pa_{x_j}\varphi(\mx) \, \beta_k  -\pa_\lambda
\big(\varphi(\mx)\zeta(\mx,\lambda)\big) \,.
\end{split}
\end{equation*} As before, the latter equation does not change the flow of the proof since we have terms of lower order on the right-hand side.

Therefore, in order to avoid proliferation of symbols, we shall assume in the rest of the proof that $h$ and $\zeta$, appearing in \eqref{q-sol-reg-1}, are compactly supported with respect to $\mx \in \R^d$. We have the same situation with the quantities $\beta\in L_{loc}^2(\R^d;{\cal M}_{loc}(\R))^d$ from \eqref{beta-es-reg} which, under the compact support assumption, actually writes as $\beta\in L^2(\R^d;{\cal M}_ {loc}(\R))^d$. 
\smallskip

\noindent \textbf{Step II.} Keeping the latter in mind, we can take the Fourier transform of \eqref{q-sol-reg-1} with respect to $\mx$ and divide such obtained relation by $|\mxi|^{1+r}$ for some fixed $r\in (0,1)$ to be determined later. We have after denoting $\mxi':=\frac{\mxi}{|\mxi|}$:  
\begin{equation}
\label{axi}
\begin{split}
&4\pi^2\langle a(\lambda) \mxi'\,|\,\frac{\mxi}{|\mxi|^{r}} \rangle \hat{h}(\mxi,\lambda) = -2\pi i\frac{\langle f(\lambda)\,|\,\mxi' \rangle}{|\mxi|^{r}} \hat{h}(\mxi,\lambda)-\frac{1}{|\mxi|^{1+r}}\pa_\lambda
\hat{\zeta}(\mxi,\lambda) \,,
\end{split}
\end{equation} 
where $\, \widehat{} \,$ denotes the Fourier transform with respect to $\mx$. The latter equation describes behavior of $\hat{h}$ on the support of $\langle a(\lambda) \mxi'\,|\,\mxi' \rangle$. 

The next step is to see what is happening out of the support of  $\langle a(\lambda) \mxi'\,|\,\mxi' \rangle$. According to the non-degeneracy conditions, this is equivalent to describing behavior of $\hat{h}$ on the support of $\langle f(\lambda)\,|\,\mxi' \rangle$.  
To this end, we use \eqref{q-sol-reg-2}. As in Section \ref{sec:velocity_averaging}, we aim to eliminate the influence of the term $\sum_{k,j=1}^d \pa_{x_j} \Bigl( \sigma_{kj}(\lambda) \, \beta_k \Bigr)$ in the latter expression. Therefore, we introduce non-negative real functions $\Phi, \Psi,\Psi_0 \in C_c^d(\R;[0,1])$, $\supp(\Psi)\subseteq (\frac{1}{2},2)$, $\supp\Psi_0\subseteq (0,2)$, such that
\begin{equation}
\label{functions}
\begin{split}
\Phi(z)=\begin{cases}
1, & |z|\leq 1\\
0, & |z|>2
\end{cases}\,,\\ 
\Psi_0(|\mxi|)+\sum\limits_{J\in \N} \Psi(2^{-J}|\mxi|)=1\,,
\end{split}
\end{equation}
(cf.~\cite[Excercise 5.1.1]{Gra}). Then, we take the Fourier transform of \eqref{q-sol-reg-2} and multiply such obtained equation by 
$$
-2\pi i\frac{\langle f(\lambda)\,|\, \mxi' \rangle}{|\mxi|} \, \Phi\Bigl(\frac{\langle a(\lambda) \mxi'\,|\,\mxi' \rangle}{\eps}\Bigr) \;,
$$  
where $\eps>0$ is arbitrary at this point.
We get
\begin{equation}
\label{fxi}
\begin{split}
&4\pi^2 \Phi\Bigl(\frac{\langle a(\lambda) \mxi'\,|\,\mxi' \rangle}{\eps}\Bigr) \,\langle {f}(\lambda)\,|\,\mxi'\rangle^2 \, \hat{h} - 
4\pi^2\Phi\Bigl(\frac{\langle a(\lambda) \mxi'\,|\,\mxi' \rangle}{\eps}\Bigr) \langle {f}(\lambda)\,|\,\mxi'\rangle \sum\limits_{k,j=1}^d  \Bigl( \sigma_{kj}(\lambda) \frac{\xi_j}{|\mxi|} \, \hat{\beta}_k \Bigr)\\
&\qquad \qquad = 2\pi i \, \frac{\langle {f}(\lambda)\,|\,\mxi'\rangle}{|\mxi|}\, \Phi\Bigl(\frac{\langle a(\lambda) \mxi'\,|\,\mxi' \rangle}{\eps}\Bigr) \pa_\lambda
\hat{\zeta}(\mxi,\lambda) \,.
\end{split}
\end{equation} 
We note that the Fourier transforms given above are well defined due to the compact support assumption.

For some $\epsilon\in (0,1)$ to be determined later and an arbitrary $J\in\N$, 
we specify $\eps$ by
$$
\eps=\eps_J:=2^{-\epsilon J} \,.
$$
Furthermore, we define 
\begin{equation*}%\label{eq:PsiJPhiJ}
\Psi_J(\mxi):=\Psi(2^{-J}|\mxi|) \qquad {\rm and} \qquad \Phi_J(\mxi,\lambda):=\Phi\Bigl(2^{\epsilon J} \langle a(\lambda) \mxi'\,|\,\mxi' \rangle\Bigr) \,.
\end{equation*}

After adding \eqref{axi} and \eqref{fxi} and multiplying such obtained expression by $\Psi_J(\mxi)$, we have
(from now on we mostly omit writing arguments of functions)
\begin{equation*}
%\label{sumfxiaxi}
\begin{split}
&4\pi^2 \Psi_J \, \Big(\langle {f}\, | \, \mxi'\rangle^2\Phi_J +\langle a\,\mxi' \, | \, \frac{\mxi}{|\mxi|^{r}} \rangle \Big) \, \hat{h}- 
 4\pi^2\Psi_J \, \langle {f}\, | \, \mxi'\rangle \Phi_J \sum\limits_{k,j=1}^d  \Bigl( \sigma_{kj} \frac{\xi_j}{|\mxi|} \, \hat{\beta}_k \Bigr)\\
&\qquad\qquad = -2\pi i\, \Psi_J \frac{\langle f\,|\,\mxi' \rangle}{|\mxi|^{r}} \hat{h}-\frac{\Psi_J}{|\mxi|^{1+r}}\pa_\lambda
\hat{\zeta}+ 2\pi i\,\Psi_J \, \frac{\langle {f}\, | \, \mxi'\rangle}{|\mxi|}\, \Phi_J \pa_\lambda
\hat{\zeta} \;.
\end{split}
\end{equation*}
Next, we denote  
\begin{equation*}%\label{eq:SymbolLJ}
{\cal L}_J(\mxi,\lambda):=\langle {f}(\lambda)\,|\,\mxi'\rangle^2 \, \Phi_J(\mxi,\lambda) +\langle a(\lambda)\mxi'\,|\,\mxi'\rangle|\mxi|^{1-r}
\end{equation*}
and sum the above identity with respect to $J\in \N$. We get (note that for any fixed $\mxi$ 
all sums are finite since at most three $\Psi_J$'s are non-zero) after dividing by $4\pi^2$:
\begin{equation}
\label{isolated-h}
\begin{split}
\biggl(\sum\limits_{J\in \N} & \Psi_J \, {\cal L}_J\biggr) \hat{h}- 
 \sum\limits_{J\in \N} \Psi_J \, \langle {f}\, | \, \mxi'\rangle \Phi_J \sum\limits_{k,j=1}^d  \Bigl( \sigma_{kj} \frac{\xi_j}{|\mxi|} \, \hat{\beta}_k \Bigr)\\
&= -\frac{i}{2\pi} \sum\limits_{J\in \N} \Psi_J \frac{\langle f\,|\,\mxi' \rangle}{|\mxi|^r} \hat{h}-\frac{1}{4\pi^2}\sum\limits_{J\in \N}\frac{\Psi_J}{|\mxi|^{1+r}}\pa_\lambda
\hat{\zeta}
	+ \frac{i}{2\pi}\sum\limits_{J\in \N}\Psi_J \, \frac{\langle {f}\, | \, \mxi'\rangle}{|\mxi|}\, \Phi_J \pa_\lambda
\hat{\zeta} \,.
\end{split}
\end{equation}
Here one can notice that $h$ is isolated on the left-hand side, which will allow us to obtain the desired bound after dividing by 
$\sum_{J\in \N}\Psi_J \, {\cal L}_J$ and multiplying by a certain power of $|\mxi|$. Notice also that the coefficient $\sum\limits_{J\in \N}  \Psi_J \, {\cal L}_J$ can be understood as a modification of the original symbol of \eqref{q-sol-reg-1} and that \eqref{isolated-h} essentially defines a new kinetic formulation of \eqref{d-p-cauchy}.

%The optimal power is dependent on the {\em quality} of the measure $\mu$, but in the worst case we know that 
% \todo{Provjeriti ostaje li $C$ po $\lambda$ i treba li biti $d\geq 2$}$\mu\in C(\R;{\cal M}_{loc}(\R^d))\subset L^1_{loc}(\R;W_{loc}^{-\alpha,r}(\R^d))$, for $\alpha \in (0,1)$ and $r\in [1,\frac{d}{d-\alpha})$. Hence, it holds 
%\begin{equation}
%\label{measure-alpha}
%\mu(\lambda,\mx)={\cal A}_{|\mxi|^\alpha} g(\cdot,\lambda)(\mx)
%\end{equation}
%for a function $g\in C(\R;L_{w,loc}^r(\R^d))$, where $L_{w,loc}^r(\R^d)$ denotes the weak topology in the space $L_{loc}^r(\R^d)$. 
\smallskip

\noindent\textbf{Step III.} At this moment we impose the following condition on parameters 
$r$ and $\epsilon$:
\begin{equation}\label{r-cond}
	r+\epsilon<1 \,.
\end{equation}

Let $K_0\geq 2$, which we specify in steps IV and V, and let us choose an arbitrary $K\in  \N$, $K\geq K_0$.
We test equation \eqref{isolated-h} with respect to $\lambda \in \R$ by 
$$
\Psi_{K}\frac{\rho(\lambda) \tilde{\Phi}\Big(\frac{\sum_{J\in \N}\Psi_J \, {\cal L}_J}{\delta}\Big)}{\sum_{J\in \N} \Psi_J {\cal L}_J} \,,
$$
where $\rho\in C^1_c(I)$, $\delta\in (0,\min\{1,\delta_0\})$ and $\tilde{\Phi}:=1-\Phi$
(later, in Step V, we specify $\delta$ and it will depend on $K$, but 
in order not to further complicate the notation, we will not explicitly state this since $K$ is at this moment fixed). 
By denoting 
$$
\tilde{\Phi}_\delta(\mxi,\lambda):=\tilde{\Phi}\biggl(\frac{\sum_{J\in \N}\Psi_J(\mxi)
	\, {\cal L}_J(\mxi,\lambda)}{\delta}\biggr) \,,
$$
we get

\begin{equation}
\label{isolated-h-0}
\begin{split}
&\int_{\R} \Psi_K \tilde{\Phi}_\delta \,  \rho \, \hat{h} \,d\lambda= 
 \sum\limits_{J=K-1}^{K+1} \int_{\R}\rho(\lambda)\frac{\tilde{\Phi}_\delta \Psi_K \Psi_J \, \langle {f}\, | \, \mxi'\rangle }{\sum\limits_{L\in \N}\Psi_L \, {\cal L}_L}\,\Phi_J \sum\limits_{k,j=1}^d  \sigma_{kj} \frac{\xi_j}{|\mxi|} \, d\hat{\beta}_k 
 \\
& \quad-\frac{i}{2\pi} \sum\limits_{J=K-1}^{K+1} \int_{\R}  \frac{\rho\,\tilde{\Phi}_\delta \Psi_K \Psi_J \frac{\langle f\,|\,\mxi' \rangle}{|\mxi|^{r}} \hat{h}}{\sum\limits_{L\in \N}\Psi_L \, {\cal L}_L}\, d\lambda
+ \frac{1}{4\pi^2}\sum\limits_{J=K-1}^{K+1} \int_{\R} \pa_\lambda \Bigl(\frac{\rho(\lambda) \tilde{\Phi}_\delta
}{\sum\limits_{L\in \N}\Psi_L \, {\cal L}_L}\Bigr)\frac{\Psi_K\Psi_J}{|\mxi|^{1+r}} \,d\hat{\zeta}(\mxi,\lambda)  \\
&\quad- \frac{i}{2\pi}\sum\limits_{J=K-1}^{K+1}\int_{\R}\pa_\lambda\Bigl(\frac{\rho\,\tilde{\Phi}_\delta  \, {\langle {f}\, | \, \mxi'\rangle}\, \Phi_J}{\sum\limits_{L\in \N}\Psi_L \, {\cal L}_L} \Bigr) 
\frac{\Psi_K \Psi_J}{|\mxi|} \,d\hat{\zeta}(\mxi,\lambda)  \;,
\end{split}
\end{equation}
where we have used that $\Psi_K\Psi_J = 0$ for $J\not\in\{K-1,K,K+1\}$.
The equality above is understood via Fourier multiplier operators. More precisely, we take the inverse Fourier transform of \eqref{isolated-h-0} with respect to $\mxi$, obtaining an equality in the sense of distributions. 
In what follows, we shall estimate the inverse Fourier transforms of each of the terms on the right-hand side of \eqref{isolated-h-0}.  
The aim is to show, with a proper choice of parameters $\epsilon, \delta$ and $r$, that there is \emph{room} to plug in $|\mxi|^s$, for some $s>0$, while keep the whole term finite when we sum up with respect to $K$. 
This will lead to the desired regularity result.

The estimation of the first two expressions on the right-hand side of the equality is carried out using Lemma \ref{L-ocjena-L1} (see also Remark 
\ref{Lp-bound}).
Thus, we need to derive estimates of the symbols in the space 
$C(\R_\lambda; C^{d+1}(\R^d_\mxi))$.
In all terms we have both $\Psi_K$ and $\rho$, hence the estimates
should be derived only for $|\mxi|\in\supp\Psi_K$ and $\lambda\in\supp\rho$.  

Since $|\mxi|\in\supp\Psi_K$ implies $\frac{1}{|\mxi|}\lesssim 2^{-K}$,
for $|\malpha|\leq d+1$ and $|\mxi|\in\supp\Psi_K$ we have (recall that $\mxi'=\mxi/|\mxi|$ and $0<\epsilon<1$):
\begin{equation*}
|\pa^\malpha_\mxi \Psi_K(\mxi)|\lesssim 2^{-|\malpha| K} \,,
	\quad |\pa^\malpha_\mxi \Phi_K(\mxi)| \lesssim 2^{-|\malpha|(1-\epsilon)K} \,,
	\quad |\pa^\malpha_\mxi\mxi'|\lesssim 2^{-|\malpha| K}\,,
\end{equation*}
where $\lesssim X$ means $\leq C X$ for a constant $C$ 
independent of parameters appearing in $X$ (sometimes we will using the index 
emphasise on which quantities the constant depends).
Hence, for $|\mxi|\in\supp\Psi_K$, $\lambda\in\supp\rho$, $1\leq |\malpha|\leq d+1$ and 
$J\in \{K-1,K,K+1\}$ it holds
\begin{equation}\label{eq:first_alpha_big}
\bigl|\pa^\malpha_\mxi\bigl(\Psi_K\Psi_J\Phi_J(\cdot,\lambda)
	\langle f(\lambda)\,|\,\cdot\,'\rangle
	\sigma(\lambda)\cdot\,'\bigr)(\mxi)\bigr|
	\lesssim 2^{-|\malpha|(1-\epsilon)K} 
	\lesssim 2^{-(1-\epsilon)K}\,.
\end{equation}
On the other hand, for $|\malpha|=0$ we use the property of 
function $\Phi_J$, together with 
$\langle a(\lambda) \mxi'\,|\,\mxi'\rangle= |\sigma(\lambda) \mxi'|^2$,
to get
\begin{equation}\label{eq:first_alpha_zero}
\bigl|\Psi_K(\mxi)\Psi_J(\mxi)\Phi_J(\mxi,\lambda)
	\langle f(\lambda) \,|\, \mxi'\rangle
	\sigma(\lambda)\mxi'\bigr|
	\lesssim 2^{-\frac{1}{2}\epsilon K} \,.
\end{equation}
%Therefore, the overall conclusion for this part of the symbol is
%\begin{equation*}
%\bigl\|\Psi_K \Psi_J \Phi_J \langle f\,|\,\mxi'\rangle \,\sigma\mxi'\bigr\|_{C(\R_\lambda;C^{d+1}(\Rd;\Rd))} \lesssim 
%	2^{-(1-\epsilon)K} + 2^{-\frac{1}{2}\epsilon K} \,.
%\end{equation*}

Concerning the first term on the right-hand side of \eqref{isolated-h-0}, we are left with 
examining $\frac{\tilde \Phi_\delta}{\sum_{L\in\N}\Psi_L \mathcal{L}_L}$. Let us present the reasoning on $\tilde\Phi_\delta$, 
and then the conclusion for $(\sum_{L\in\N} \Psi_L\mathcal{L}_L)^{-1}$ will follow analogously. 
Obviously $\tilde \Phi_\delta(\mxi,\lambda)$ is bounded, i.e.~of the order 1. Further on, for the first order derivative we have
\begin{equation*}
\pa_{\xi_j} \tilde\Phi_\delta(\mxi,\lambda)
	= \frac{1}{\delta}\tilde\Phi' \Bigl(\frac{\sum_L \Psi_L\mathcal{L}_L}{\delta}\Bigr)
	\sum_{L=K-1}^{K+1} \Bigl((\pa_{\xi_j}\Psi_L)(\mxi)\mathcal{L}_L(\mxi,\lambda)
	+\Psi_L(\mxi)(\pa_{\xi_j}\mathcal{L}_L)(\mxi,\lambda)\Bigr) \,.
\end{equation*}
The expression under the summation sign is of the order $2^{-r K}$
(note that $|\mxi|^{1-r}$ appears in $\mathcal{L}_L(\mxi,\lambda)$). 
Hence, $\pa_{\xi_i} \tilde\Phi_\delta(\mxi,\lambda)$ is of the order
$\frac{1}{\delta}2^{-r K}$.
Now, each time a derivative hits the expression under the summaton sign we get $2^{-K}$, while the same with $\tilde\Phi'(\frac{\sum_L\Psi_L \mathcal{L}_L}{\delta})$ results 
in getting another $\frac{1}{\delta}2^{-r K}$ in the order. 
Since $\frac{1}{\delta}2^{-r K} \geq 2^{-K}$ (as $\delta, r\leq1$), 
we have
\begin{equation*}
\bigr|\pa^\malpha_\mxi\tilde \Phi_\delta (\mxi,\lambda)\bigl| 
	\lesssim \frac{1}{\delta^{|\malpha|}}2^{-r|\malpha| K} \;.
\end{equation*}
%Here we introduce a condition between $\delta, r$ and $K$:
%\begin{equation}\label{eq:delta_cond1}
%\delta > 2^{-rK} \,,
%\end{equation}
%implying that the worst case in the previous inequality is attained 
%for $|\malpha|=0$, i.e.~for any $\malpha\in\N_0^d$ it holds
%$\bigr|\pa^\malpha_\mxi\tilde \Phi_\delta (\mxi,\lambda)\bigl| 
%\lesssim 1$.
Analogously, 
\begin{equation*}
\biggr|\pa^\malpha_\mxi\biggl(\frac{1}{\sum_{L=K-1}^{K+1}\Psi_L
	\mathcal{L}_L}\biggr)(\mxi,\lambda)\biggl| 
	\lesssim \frac{1}{\delta^{1+|\malpha|}}2^{-r|\malpha| K}
%	\lesssim \frac{1}{\delta} 
\end{equation*}
(here we used that $(\mxi,\lambda)\in\supp\tilde\Phi_\delta$).
Thus, using the Leibniz product rule, for $0\leq |\malpha|\leq d+1$ we get
\begin{equation}\label{eq:first_tilde}
\biggr|\pa^\malpha_\mxi\biggl(\frac{\tilde\Phi_\delta}
	{\sum_{L=K-1}^{K+1}\Psi_L
	\mathcal{L}_L}\biggr)(\mxi,\lambda)\biggl| 
\lesssim \frac{1}{\delta^{1+|\malpha|}}2^{-r|\malpha| K} 
	\lesssim \frac{1}{\delta}\Bigl(1+\frac{1}{\delta^{d+1}}2^{-(d+1)rK}\Bigr)\;.
%\lesssim \frac{1}{\delta} \;.
\end{equation}

Combining \eqref{eq:first_alpha_big}, \eqref{eq:first_alpha_zero} and \eqref{eq:first_tilde}, 
together with the condition \eqref{r-cond}, Lemma \ref{L-ocjena-L1} provides the 
following bound corresponding to the first expression on the
right-hand side of \eqref{isolated-h-0}
(denoting $\beta=(\beta_1,\dots,\beta_d)$):
\begin{equation}
\label{beta}
\begin{split}
\biggl\| {\cal F}^{-1}\Big(\int_{\R} & \rho(\lambda)  \langle {f}(\lambda)\, | \, \mxi'\rangle  \,\frac{\tilde{\Phi}_\delta(\cdot,\lambda)\,  }{\sum\limits_{L\in \N}\Psi_L \, {\cal L}_L(\cdot,\lambda)}\,\Psi_K\, \Psi_J \, \Phi_J(\cdot,\lambda)  \sum\limits_{k,j=1}^d  \sigma_{kj}(\lambda) \frac{\xi_j}{|\mxi|} \, d\hat{\beta}_k\Big) \biggr\|_{L^1(\R^d)}\\
&\lesssim_{\rho,f,a,\beta} \frac{1}{\delta}\Bigl(2^{-(1-\epsilon)K}+2^{-\frac{1}{2}\epsilon K}\Bigr)
	\Bigl(1+\frac{1}{\delta^{d+1}}2^{-(d+1)rK}\Bigr) \,.
%&\lesssim_{\rho,f,\beta} \frac{1}{\delta}\Bigl(2^{-(1-\epsilon)K}+2^{-\frac{1}{2}\epsilon K}\Bigr) \,.
\end{split}
\end{equation}

Similarly, the group of second terms on the right-hand side of \eqref{isolated-h-0} is estimated as follows (see Remark \ref{Lp-bound}):
\begin{equation}
\label{2}
\begin{split}
\biggl\| {\cal F}^{-1}\Big(\int_{\R} \rho(\lambda) \frac{\tilde{\Phi}_\delta(\cdot,\lambda)\, \Psi_K\, \Psi_J \langle f(\lambda)\,|\,\mxi' \rangle \hat{h}(\cdot,\lambda)}{|\mxi|^{r} \sum\limits_{L\in \N}\Psi_L \, {\cal L}_L(\cdot,\lambda)} & \, d\lambda \Big)\biggr\|_{L^1(\R^d)}\\
&\lesssim_{\rho,f,a} \frac{2^{-rK}}{\delta} \Bigl(1+\frac{1}{\delta^{d+1}}2^{-(d+1)rK}\Bigr) \,.
%\lesssim_{\rho,f} \frac{2^{-rK}}{\delta} \,.
\end{split}
\end{equation} 
 
It remains to deal with the final two terms which we collect together. Thus, the symbol reads $\frac{1}{4\pi^2}\sum_{J=K-1}^{K+1} \Theta_J$,
where
\begin{equation*}
\Theta_{J}(\mxi,\lambda) :=\frac{\Psi_K(\mxi) \Psi_J(\mxi)}{|\mxi|} \, 
	\pa_\lambda \Biggl(\frac{\rho(\lambda) \tilde{\Phi}_\delta(\mxi,\lambda) }
	{\sum_{L=K-1}^{K+1}\Psi_L(\mxi) \, {\cal L}_L(\mxi,\lambda)}
	\Bigl(|\mxi|^{-r} - 2\pi i \,\Phi_J(\mxi,\lambda) \, \langle {f}(\lambda)\, | \, \mxi'\rangle\Bigr) \Biggr)
\end{equation*} 
(we used the fact that $\Psi_K$ and $\Psi_J$, $J\notin \{K-1,K,K+1\}$, have disjoint supports). 
After derivation we get (we do not write arguments of function, while
$'$ denotes the derivative of the function; with the exception of 
$\mxi'=\mxi/|\mxi|$):
\begin{equation}
\label{Mcz-1}
\begin{split}
\Theta_{J} &=\frac{\Psi_K\,\Psi_J}{|\mxi|} 
	\Biggl(\frac{\rho' \,\tilde{\Phi}_\delta }
	{\sum_{L}\Psi_L\, {\cal L}_L}
	\Bigl(|\mxi|^{-r} - 2\pi i \,\Phi_J \, \langle {f}\, | \, \mxi'\rangle\Bigr) \\
&\qquad + \frac{\rho \,\frac{1}{\delta}\tilde{\Phi}'\,
	\bigl(\sum_L \Psi_L(\pa_\lambda\mathcal{L}_L)\bigr) }
{\sum_{L}\Psi_L \, {\cal L}_L}
\Bigl(|\mxi|^{-r} - 2\pi i \,\Phi_J \, \langle {f}\, | \, \mxi'\rangle\Bigr) \\
&\qquad - \frac{\rho \,\tilde{\Phi}_\delta\,
	\bigl(\sum_L \Psi_L(\pa_\lambda\mathcal{L}_L)\bigr) }
{\bigl(\sum_{L}\Psi_L \, {\cal L}_L\bigr)^2}
\Bigl(|\mxi|^{-r} - 2\pi i \,\Phi_J \, \langle {f}\, | \, \mxi'\rangle\Bigr) \\
&\qquad - \frac{2\pi i\,\rho \,\tilde{\Phi}_\delta }
{\sum_{L}\Psi_L\, {\cal L}_L}
\Bigl(2^{\epsilon J}\Phi' \, 
	\langle a'\,\mxi'\,|\,\mxi'\rangle \, \langle {f}\, | \, \mxi'\rangle
	+ \Phi_J \, \langle {f'}\, | \, \mxi'\rangle\Bigr)\Biggr)  \,.
\end{split}
\end{equation}
According to Lemma \ref{l-meas} and Remark \ref{l-meas-rem}, 
we have:
\begin{equation}
\label{g-term}
\biggl\| \mathcal{F}^{-1}\Bigl(\int_\R \Theta_J(\cdot,\lambda)\,
	d\hat\zeta(\cdot,\lambda)\Bigr)\biggr\|_{\mathcal{M}(\Rd)}
	= \bigl\| \mathcal{A}_{\Theta_J}\zeta\bigr\|_{L^1(\R_\lambda;\mathcal{M}(\Rd))}
	\lesssim_{\rho,\zeta} \| \Theta_{J} \|_{C(\R; C^{d+1}(\R^d))} \,,
\end{equation} 
and it remains to estimate the summands in \eqref{Mcz-1} with respect to the given bound.

If we notice that 
\begin{align*}
\biggl|\partial^\malpha_\mxi \biggl(\sum\limits_{L=K-1}^{K+1}\Psi_L \, 
	(\pa_\lambda {\cal L}_L)\biggr)(\mxi,\lambda)\biggr| \lesssim {2^{(1-r) K}}
\end{align*} 
and take into account properties of functions 
$\Psi_J$, $\tilde{\Phi}_\delta$, and $\Phi_J$
(in particular, the estimates \eqref{eq:first_alpha_big},
\eqref{eq:first_alpha_zero} and \eqref{eq:first_tilde}), 
we have from \eqref{Mcz-1} and \eqref{g-term}
\begin{equation}
\label{g-term-1}
\sum_{J=K-1}^{K+1}\bigl\|\mathcal{A}_{\Theta_J}\zeta\bigr\|_{L^1(\R_\lambda;
	\mathcal{M}(\Rd))}
\lesssim_{\rho, f,a, \zeta} \frac{1}{\delta} \Bigl(2^{-(1-\epsilon)K}
	+ \frac{1}{\delta} 2^{-r K}\Bigr)\Bigl(1+\frac{1}{\delta^{d+1}}2^{-(d+1)rK}\Bigr) \,.
%\lesssim_{\rho, f,a, \zeta} \frac{1}{\delta} 2^{-(1-\epsilon)K}
%	+ \frac{1}{\delta^2} 2^{-r K} \,.
\end{equation} 

Therefore, combining \eqref{beta}, \eqref{2} and \eqref{g-term-1}
we have
\begin{equation*}
\begin{split}
\Bigl\| \int_{\R} {\cal A}_{\Psi_K \tilde{\Phi}_{\delta}(\cdot,\lambda)}  
	& \bigl( \rho(\lambda) {h}(\cdot,\lambda)\bigr) \, d\lambda \Bigr\|_{L^1(\R^d)}\\
&\lesssim_{\rho, f,a, \beta,\zeta} \frac{1}{\delta}\Bigl(2^{-\frac{1}{2}\epsilon K}+2^{-(1-\epsilon)K}
	+ \frac{1}{\delta} 2^{-r K}\Bigr)\Bigl(1+\frac{1}{\delta^{d+1}}2^{-(d+1)rK}\Bigr) \,.
%	\lesssim \frac{1}{\delta} 2^{-(1-\epsilon)K} + \frac{1}{\delta}2^{-\frac{1}{2}\epsilon K} 
%	+ \frac{1}{\delta^2} 2^{-r K} \,,
\end{split}
\end{equation*}
where we have used that $0<\delta<1$. Furthermore, using \eqref{r-cond}
we have $\delta^{-1}2^{-(1-\epsilon)K}\leq \delta^{-2}2^{-rK}$. Thus, the previous estimate simplifies to
\begin{equation}\label{final-regular}
\begin{split}
\Bigl\| \int_{\R} {\cal A}_{\Psi_K \tilde{\Phi}_{\delta}(\cdot,\lambda)}  
 & \bigl( \rho(\lambda) {h}(\cdot,\lambda)\bigr)\, d\lambda \Bigr\|_{L^1(\R^d)} \\
&\lesssim_{\rho, f,a, \beta,\zeta} \Bigl(\frac{1}{\delta}2^{-\frac{1}{2}\epsilon K}+ \frac{1}{\delta^2} 2^{-r K}\Bigr)\Bigl(1+\frac{1}{\delta^{d+1}}2^{-(d+1)rK}\Bigr) \,.
%\lesssim  \frac{1}{\delta}2^{-\frac{1}{2}\epsilon K} 
%+ \frac{1}{\delta^2} 2^{-r K} \,.
\end{split}
\end{equation}

%Now, we want to choose the parameters $\epsilon, \eta'$, and $\eta''$ such that the sum of \eqref{beta}, \eqref{2}, and \eqref{g-term-1} over $k,J\in \N$ remains bounded. Clearly, by  \eqref{g-term-1}, we need to have 
%\begin{equation}
%\label{eps-cond}
% \epsilon<\frac{1}{1+d}.
%\end{equation} Now, we denote
%$$
%\tilde{\eps}=\min\big\{\frac{1 -(d+1)\epsilon }{d+1}, \, r  \big\}
%$$ and take
%$$
%\delta=\delta_k:=\min \big\{ \frac{1}{2},2^{(\eps-\tilde{\eps})k} \big\}
%$$ for a fixed $0<\eps<\tilde{\eps}$.
%
%We have thus handled the behavior of $\int_{\R} \rho{\cal A}_{|\mxi|^{\eta'}} h(\cdot,\lambda)  d\lambda $ on the dyadic blocks when the transformed symbol $\sum\limits_{J\in \N} {\cal L}_J \Psi_J$ is strictly greater than zero. In other words, using \eqref{beta}, \eqref{2}, and \eqref{g-term-1}, we have from \eqref{isolated-h-0}
%\begin{equation}
%\label{final-regular}
%\| \int_{\R} {\cal A}_{\Psi_k \tilde{\Phi}_{\delta_k}}  \big( \rho(\lambda) {h}(\cdot,\lambda) d\lambda\big) \|_{L^1(\R^d)} \lesssim 2^{-\eps k}.
%\end{equation} Not in passing that from here it follows that the Sobolev regularity of the velocity average $\int \rho(\lambda) {h}(\cdot,\lambda) d\lambda$ cannot be greater than $\frac{1}{d+1}$ (using this method).

\smallskip

\noindent\textbf{Step IV.} Now we turn to the situation when the modified symbol $\sum_{J\in \N} {\cal L}_J \Psi_J$ is close to zero, which is achieved by studying the symbol 
$1-\tilde\Phi_\delta$. 
Notice that a multiplier operator with a smooth compactly supported symbol $\psi$ generates an infinitely differentiable function, i.e. ${\cal A}_\psi: L^1(\R^d)\to C^\infty(\R^d)$. Thus, it is enough to inspect behavior of ${\cal A}_{1-\tilde{\Phi}_\delta}$ out of an arbitrary neighbourhood of the origin.  
This is the reason why we consider only functions $\Psi_K$ for $K$'s which are greater or equal than $K_0\geq 2$. 
The estimate is obtained by means of the non-degeneracy condition 
\eqref{non-deg-stand} and the condition \eqref{r-cond} will play an important role. 

For $K\geq K_0$ we have 
\begin{equation}
\label{aroundzero}
\begin{split}
\Bigl\|\int_{\R} & {\cal A}_{\Psi_K (1-\tilde{\Phi}_{\delta}(\cdot,\lambda))}
	\bigl(\rho(\lambda)h(\cdot,\lambda)\bigr) \, d\lambda \Bigr\|_{L^2(\R^d)}
	=\Bigl\| \, \Psi_K \, \int_{\R}(1-\tilde{\Phi}_\delta(\cdot,\lambda)) \rho(\lambda)\hat{h}(\cdot,\lambda)  \, d\lambda \Bigr\|_{L^2(\R^d)}\\
&\lesssim \sup\limits_{|\mxi| \sim 2^K}\operatorname{meas}\Bigl\{\lambda\in
	\supp(\rho):\, \sum\limits_{J\in \N} \Psi_J(\mxi) {\cal L}_J(\mxi,\lambda) 
	\leq 2\delta\Bigr\}^{1/2} \|h\|_{L^2(\R^{d+1})} \\
&\lesssim \sup\limits_{|\mxi|\sim 2^K} 
	\operatorname{meas}\Bigl\{\lambda \in \supp(\rho) :\, \big|\langle f(\lambda)\,|\, {\mxi'}
	\rangle \big|^2 + \langle a(\lambda){\mxi'} \,|\, {\mxi'} \rangle |\mxi|^{1-r}  \leq  2\delta \Bigr\}^{1/2}\, \|h\|_{L^2(\R^{d+1})}\\
&\lesssim \sup\limits_{|\mxi|=1} 
\operatorname{meas}\Bigl\{\lambda \in I :\, \big|\langle f(\lambda)\,|\, {\mxi'}
\rangle \big|^2 + \langle a(\lambda){\mxi'} \,|\, {\mxi'} \rangle  \leq  2\delta \Bigr\}^{1/2}\, \|h\|_{L^2(\R^{d+1})}\\
& \lesssim  \delta^{\alpha/2} \|h\|_{L^2(\R^{d+1})} \,,
\end{split}
\end{equation} 
where we recall that $\mxi'=\mxi/|\mxi|$. In the first two inequalities we have used properties of functions 
$\tilde\Phi_\delta$ and $\Psi_K$, while the third one follows from the fact that $\supp(\rho)\subseteq I$ 
and $|\mxi|^{1-r}\geq 1$ (since $1-r>0$ and $2^K> 1$).
The last step is a direct application of the
non-degeneracy condition \eqref{non-deg-stand}. Let us discuss the second 
inequality in more details. We take arbitrary $\mxi\in\supp\Psi_K$ and
$\lambda\in\supp(\rho)$ such that 
\begin{equation*}
%\label{nd-final-1}
\begin{aligned}
2\delta \geq
	\sum\limits_{J\in \N} \Psi_J(\mxi) {\cal L}_J(\mxi,\lambda)
	&=\bigl|\langle f'(\lambda)\, | \, \mxi' \rangle\bigr|^2 \sum\limits_{J\in \N} \Psi_J(\mxi)\Phi_J(\mxi,\lambda) +\langle a(\lambda)\mxi'\,| \, \frac{\mxi}{|\mxi|^r} \rangle \sum\limits_{J\in\N}\Psi_J(\mxi) \\
&= \bigl|\langle f'(\lambda)\, | \, \mxi' \rangle\bigr|^2 \sum\limits_{J=K-1}^{K+1} \Psi_J(\mxi)\Phi_J(\mxi,\lambda) +\langle a(\lambda)\mxi'\,| \, \mxi' \rangle |\mxi|^{1-r} \,,
\end{aligned}
\end{equation*}
and the last equality is due to the fact that $\sum_{J\in \N_0}\Psi_j=1$ (note that $\Psi_K\Psi_0=0$ since $K\geq 2$).
Hence (recall that $\mxi\in\supp\Psi_K$ and $r<1$), 	
\begin{equation}
\label{est-1}
\langle a(\lambda)\mxi'\,| \, \mxi' \rangle \leq 2\delta |\mxi|^{r-1} 
	\leq 2\delta\, 2^{(r-1)(K-1)}\,.
\end{equation} 

It is sufficient to prove that for $K$ large enough $\Phi_J$, $J\in\{K-1,K,K+1\}$, is identically
equal to 1 on $\supp\Psi_K\times\supp\rho$. 
Indeed, in that case we would have
\begin{equation*}
\bigl|\langle f'(\lambda)\, | \, \mxi' \rangle\bigr|^2 \sum\limits_{J=K-1}^{K+1} \Psi_J(\mxi)\Phi_J(\mxi,\lambda)
	= \bigl|\langle f'(\lambda)\, | \, \mxi' \rangle\bigr|^2 \sum\limits_{J=K-1}^{K+1} \Psi_J(\mxi)
	= \bigl|\langle f'(\lambda)\, | \, \mxi' \rangle\bigr|^2 \,.
\end{equation*}
To this end, we just need to study $\Phi_{K+1}$ (the worst case concerning what we want to show).
By \eqref{est-1} we have
$$
|2^{\epsilon (K+1)}\langle a(\lambda)\mxi'\,| \, \mxi' \rangle| 
	\leq 2\delta \, 2^{-K(1-r-\epsilon)} \, 2^{1-r+\epsilon}
	\leq 8 \cdot 2^{-K_0(1-r- \epsilon)} \,,
$$ 
where we have used that $\delta, r, \epsilon\in (0,1)$.
Since by \eqref{r-cond} we have $1-r-\epsilon>0$, for fixed $r$ and $\epsilon$
(they depend only on $\alpha$; see Step V)
we can choose $K_0$ large enough so that 
the right-hand side in the inequality above is less than 1.
Thus, by \eqref{functions} we have $\Phi_{K+1}(\mxi,\lambda)=1$.

\smallskip

\noindent\textbf{Step V.} From \eqref{final-regular} and \eqref{aroundzero} we are in position to use the K-interpolation method. The procedure is fairly standard (it is described in detail in \cite[pp.~2516--2518]{GH}), but for the completeness and to 
recognise the right order of the smoothing effect, we present the argument in 
a full detail. 

In this final technical part we finally specify parameters $r$ and $\epsilon$. 
For a positive and sufficiently small constant $c$ we define
\begin{equation}\label{eq:r-epsilon}
r:= \frac{\alpha+4}{3\alpha+8} - c\qquad \hbox{and} \qquad
	\epsilon:=\frac{2\alpha+4}{\alpha+4}r+2c \;.
\end{equation}
Note that with this choice the condition \eqref{r-cond} is fulfilled.

We denote
$$
H_J(\mx) :=\int_{\R} \rho(\lambda) {\cal A}_{\Psi_J}(h(\cdot,\lambda))(\mx)d\lambda\,, \ \ J\in \N \,,
$$
and introduce
$$
\mathcal{K}(\tau,H_K) := \inf\limits_{\substack{F_K^0\in L^2(\R^d), \, F_K^1\in L^1(\R^d)\\ F_K^0+F_K^1=H_K}}\Bigl( \|F_K^0\|_{L^2(\R^d)}+\tau\|F_K^1\|_{L^1(\R^d)} \Bigr) \,.
$$

By choosing $F^0_K=H_K$ and $F^1_K=0$ we have a trivial 
estimate of $\mathcal{K}(\tau, H_K)$ that is independent 
of $\tau$:
\begin{equation}\label{eq:KnormL2}
\mathcal{K}(\tau,H_K) 
	\leq \|\rho h\|_{L^2(\R^{d+1})} \,.
\end{equation}

On the other hand, for 
\begin{align*}
F_K^0 &= \int_{\R} {\cal A}_{\Psi_K (1-\tilde{\Phi}_{\delta}(\cdot,\lambda))}
\bigl(\rho(\lambda)h(\cdot,\lambda)\bigr) \, d\lambda \,,\\
F_K^1 &=\int_{\R} {\cal A}_{\Psi_K \tilde{\Phi}_{\delta}(\cdot,\lambda)}
\bigl(\rho(\lambda)h(\cdot,\lambda)\bigr) \, d\lambda \,,
\end{align*}
by \eqref{final-regular} and \eqref{aroundzero} we get
\begin{equation*}
\mathcal{K}(\tau,H_K) \lesssim \delta^{\alpha/2} +
	\tau\Bigl(\frac{1}{\delta}2^{-\frac{1}{2}\epsilon K}+ \frac{1}{\delta^2} 2^{-r K}\Bigr)\Bigl(1+\frac{1}{\delta^{d+1}}2^{-(d+1)rK}\Bigr)\;.
%	\tau\Bigl(\frac{1}{\delta}2^{-\frac{1}{2}\epsilon K} + \frac{1}{\delta^2} 2^{-r K}\Bigr) \;.
\end{equation*}
For a fixed $\tau$ and $K$, let us take $\delta=\tau^\eta 2^{-\omega K}$
(the only requirement on $\delta$ form the previous steps is $\delta\in(0,\min\{1),\delta_0\}$ on which we will comment later), where we choose
$\eta,\omega>0$ in an optimal way. 
The previous estimate then reads
\begin{equation}\label{eq:Knorm_tau1}
\begin{split}
\mathcal{K}(\tau,H_K) \lesssim \tau^{\frac{\alpha\eta}{2}}2^{-\frac{\omega\alpha}{2}K} 
	&+\tau^{1-\eta}2^{-(\frac{\epsilon}{2}-\omega)K}
	+\tau^{1-(d+2)\eta}2^{-\left(\frac{\epsilon}{2}-\omega+(d+1)(r-\omega)\right)K} \\
&+\tau^{1-2\eta} 2^{-(r-2\omega)K}
	+\tau^{1-(d+3)\eta} 2^{-\left(r-2\omega+(d+1)(r-\omega)\right)K}  \;.
\end{split}
\end{equation}

Now we choose $\eta$ and $\omega$ to maximise the minimum of the corresponding exponents, i.e.~we are solving
\begin{equation*}
\min\Bigl\{\frac{\alpha\eta}{2},1-\eta,1-(d+2)\eta, 1-2\eta, 1-(d+3)\eta\Bigr\} \longrightarrow \max
\end{equation*}
and
\begin{equation*}
\min\Bigl\{\frac{\omega\alpha}{2}, \frac{\epsilon}{2}-\omega,\frac{\epsilon}{2}-\omega+(d+1)(r-\omega),r-2\omega,r-2\omega+(d+1)(r-\omega)\Bigr\}
	\longrightarrow \max \,.
\end{equation*}
For $\eta$ we get $\eta=\frac{2}{\alpha+2(d+3)}$ (achieved by solving $\frac{\alpha\eta}{2}=1-(d+3)\eta$).
On the other hand, since obviously it should be $\omega<r$, we can omit the third and the fifth expression
in the minimisation problem for $\omega$. 
Moreover, the choice of parameters $r$ and $\epsilon$ made in \eqref{eq:r-epsilon}
ensures that the optimal $\omega$ is given by $\omega=\frac{2r}{\alpha+4}$. 

For the chosen values of $\eta$ and $\omega$, 
\eqref{eq:Knorm_tau1} reads
\begin{equation}\label{eq:Knorm_tau2}
\begin{split}
\tau^{-\theta_*} 2^{\frac{r\alpha}{\alpha+4}K}\mathcal{K}(\tau,H_K) 
	\lesssim 1 &+\tau^{\frac{2d+4}{\alpha+2(d+3)}} 2^{-c K} 
	+\tau^{\frac{2}{\alpha+2(d+3)}} 2^{-\left(c+(d+1)\frac{\alpha+2}{\alpha+4}r\right) K} \\
&+ \tau^{\frac{2d+2}{\alpha+2(d+3)}}
	+ 2^{-(d+1)\frac{\alpha+2}{\alpha+4}r K} \,,
\end{split}
\end{equation}
where $\theta_*:=\frac{\alpha}{\alpha+2(d+3)}$ and $c$ is given in \eqref{eq:r-epsilon}. 

The previous estimate is valid for any $\tau>0$ for which $\delta<\min\{1,\delta_0\}$. 
Let us take $T>0$ (to be determined later) such that for any $\tau\in (0,T)$ we have 
$\delta=\delta(\tau)<\min\{1,\delta_0\}$. 
Now, by applying the K-method of real interpolation we get
\begin{equation}\label{eq:reg_Kint_T}
\begin{split}
2^{\frac{r\alpha}{\alpha+4}K} & \|H_K\|_{(L^2,L^1)_{\theta_*,\infty}}
	\leq 2^{\frac{r\alpha}{\alpha+4}K}
	\bigl\|\tau^{-\theta_*}\mathcal{K}(\tau,H_K)
	\bigr\|_{L^\infty((0,\infty)_\tau)} \\
&\leq 2^{\frac{r\alpha}{\alpha+4}K}
\bigl\|\tau^{-\theta_*}\mathcal{K}(\tau,H_K)
\bigr\|_{L^\infty((0,T)_\tau)}+
2^{\frac{r\alpha}{\alpha+4}K}
\bigl\|\tau^{-\theta_*}\mathcal{K}(\tau,H_K)
\bigr\|_{L^\infty([T,\infty)_\tau)} \\
&\lesssim 1 +T^{\frac{2d+4}{\alpha+2(d+3)}} 2^{-c K} 
	+T^{\frac{2}{\alpha+2(d+3)}} 2^{-\left(c+(d+1)\frac{\alpha+2}{\alpha+4}r\right) K}
	+ T^{\frac{2d+2}{\alpha+2(d+3)}} \\
&\qquad + 2^{-(d+1)\frac{\alpha+2}{\alpha+4}r K}
	+ T^{-\theta_*}2^{\frac{r\alpha}{\alpha+4}K} \,,
\end{split}
\end{equation}
where in the last inequality we have used 
\eqref{eq:KnormL2} and \eqref{eq:Knorm_tau2}.

Let us now specify $T>0$. We take $T=2^{\gamma K}$ for $\gamma>0$ to be determined in an optimal way. 
More precisely, plugging this expression for $T$ into the previous estimate, 
we want to optimise all exponents containing $\gamma$, 
i.e.~we want that the maximum of
\begin{align*}
&\frac{2d+4}{\alpha+2(d+3)}\gamma-c \;,
	& &\frac{2}{\alpha+2(d+3)}\gamma-c+(d+1)\frac{\alpha+2}{\alpha+4}r \;,\\
&\frac{2d+2}{\alpha+2(d+3)}\gamma \;,
	& &\frac{r\alpha}{\alpha+4}-\frac{\alpha}{\alpha+2(d+3)}\gamma \;,
\end{align*}
is the smallest possible.
One can show that under the condition 
\begin{equation}\label{eq:c-cond}
	c\leq\frac{2\alpha(\alpha+4)}{(3\alpha+8)\bigl((\alpha+4)(\alpha+2(d+1))+2\alpha\bigr)}=:c_*
\end{equation}
the solution is 
\begin{equation*}
\gamma = \frac{\alpha+2(d+3)}{\alpha+2(d+2)}\Bigl(c+\frac{r\alpha}{\alpha+4}\Bigr) \,.
\end{equation*}

Let us first check that such $T$ is admissible, 
i.e.~that $T^\eta 2^{-\omega K}= 2^{-(\omega-\gamma\eta)}<1$. 
It is sufficient to show that $\omega-\gamma\eta>0$ (since we can take $K_0$ large enough). 
Taking into account values of $\omega, \eta, \gamma$ and $r$, we have
\begin{align*}
\omega-\gamma\eta &= \frac{2r}{\alpha+4} - 	\frac{2}{\alpha+2(d+3)}\frac{\alpha+2(d+3)}{\alpha+2(d+2)}\Bigl(c+\frac{r\alpha}{\alpha+4}\Bigr) \\
&=\frac{4(d+2)}{(\alpha+2(d+2))(3\alpha+8)} - \frac{2\alpha+4(d+4)}{(\alpha+2(d+2))(\alpha+4)}c \\
&\geq \frac{4}{(\alpha+2(d+2))(3\alpha+8)}\Bigl(d+2- \frac{\alpha(\alpha+2(d+4))}{(\alpha+2(d+1))(\alpha+4)+2\alpha}\Bigr)\,,
\end{align*} 
where in the last line we have applied \eqref{eq:c-cond}. 
Since the final expression is strictly positive, we indeed have 
$\omega-\gamma\eta>0$ for any positive $c$ satisfying \eqref{eq:c-cond}.
Moreover, since for $c=c_*$ we clearly have that $r$ is positive 
(see \eqref{eq:r-epsilon} and \eqref{eq:c-cond}), the value $c_*$ is 
really admissible for $c$. 

Returning to \eqref{eq:reg_Kint_T}, it is easy to see that for the obtained $\gamma$ 
the maximal value of the four expressions given above is 
\begin{equation*}
	\frac{2(d+2)}{\alpha+2(d+2)}\frac{r\alpha}{\alpha+4}-\frac{\alpha}{\alpha+2(d+2)}c \,,
\end{equation*}
which is obviously less than $\frac{r\alpha}{\alpha+4}$ (the exponent appearing on the 
left-hand side of \eqref{eq:reg_Kint_T}).
Therefore, after dividing \eqref{eq:reg_Kint_T} by $2^{\frac{r\alpha}{\alpha+4}K}$
the highest exponent on the right-hand side is given by
\begin{align*}
- \frac{\alpha}{\alpha+2(d+2)} \Bigl(\frac{r\alpha}{\alpha+4}+c\Bigr) 
	&= - \frac{\alpha}{\alpha+2(d+2)} \Bigl(\frac{\alpha}{3\alpha+8}+\frac{4}{\alpha+4}c\Bigr) \\
&\geq -\frac{\alpha^2(\alpha+4)}{(3\alpha+8)\bigl((\alpha+4)(\alpha+2(d+1))+2\alpha\bigr)}=:-s_* \,,
\end{align*}
where we took into account \eqref{eq:c-cond}. Therefore, it is optimal to take $c=c_*$.
Then from \eqref{eq:reg_Kint_T} we finally get
\begin{equation}\label{eq:interp_norm}
\|H_K\|_{(L^2,L^1)_{\theta_*,\infty}}
\lesssim 2^{-s_* K} \,.
\end{equation}

\smallskip

\noindent\textbf{Step VI.} It is left to follow 
the Conclusion of the proof of \cite[Theorem 3.1]{GH}.
More precisely, let us note first that $(L^2(\R^d),L^1(\R^d))_{\theta_*,\infty} $ is actually the Lorentz space which is a subset of the Lebesgue space $L^q(\R^d)$ for any $q>1$ such that
$\frac{1}{q}>\frac{1-\theta_*}{2}+\theta_*$ (see \cite[page 2578]{GH} and \cite[Section 1.18.6]{triebel};
note that by Step I of the proof $h$ is compactly supported with respect to $\mx$).
Then, using \eqref{eq:interp_norm}, it is left to sum with respect to $K\in\N$ the norms 
of $H_K$'s in the space $W^{s,q}(\R^d)$ for $q$ given above and $s<s_*$, obtaining
$$
\int_\R h(\cdot,\lambda) \rho(\lambda)\,d\lambda \in W^{s,q}(\R^d) \,.
$$ 

Going back to Step I, let us recall that $h$ is localised with respect to the variable $\mx$.
Thus, for the original solution $h$ to \eqref{q-sol-reg-1}, the result is the same but
for the loc-space $W_{loc}^{s,q}(\R^d)$.
The proof is over. 
\end{proof}

%We now use the standard $K$-interpolation method \cite{BL-book}. We denote
%\begin{equation*}
%\begin{split}
%\overline{f_0}(\mx)=\int_{\R} f_0(\mx,\lambda) \rho(\lambda) d\lambda\\
%\overline{f_1}(\mx)=\int_{\R}f_1(\mx,\lambda) \rho(\lambda) d\lambda
%\end{split}
%\end{equation*} By choosing appropriate $\theta>0$, we deduce from \eqref{final-regular} and \eqref{aroundzero} as in \cite[page 17]{GH}
%\begin{align*}
%&\inf_{\overline{f_0}(\mx)+\overline{f_1}(\mx)=\int_{\R}h(\mx,\lambda)\rho(\lambda)d\lambda}\big[\|\overline{f_0} \|_{\dot{W}^{\eps,2}(\R^d)(\R^d)}+t\|\overline{f_1}\|_{{\dot{W}^{\eta',1}(\R^d)}} \big]
%\\&\leq t^\theta (\|\beta\|_{L^1(\R^d;{\cal M}(\R))}+\|h\|_{L^1(\R^d\times \R)}+ \|\mu\|_{{\cal M}(\R^{d+1}}+1)^\gamma
%\end{align*} for some $\gamma>0$. From here, the standard interpolation argument (see e.g. \cite[page 153]{BL-book}) concludes the theorem. More precisely, by repeating the arguments from \cite[page 1496]{TT}, we get $u\in W_{loc}^{\frac{\tilde{\alpha} \eta'}{\tilde{\alpha}+2},1}(\R^d)$.

Let us illustrate the applicability of the previous theorem on the example given in 
\cite[Corollary 4.5]{TT}, where the result was not completely satisfactory as the 
regularising effect was not shown for all values of parameters (only for $n\geq 2l$).
We are able to prove the following.

\begin{corollary}
	\label{TT-ex} Let $n,l\in\N$. The two dimensional convection-diffusion equation 
	\begin{equation}
	\label{TT-1}
	\begin{split}
	\pa_t u(t,\mx) &+\left( \frac{\pa}{\pa x_1}+\frac{\pa}{\pa x_2} \right) \left(\frac{1}{l+1} u^{l+1}(t,\mx)\right)\\
	&-\left(\frac{\pa^2}{\pa x_1^2}-2\frac{\pa^2}{\pa x_1 \pa x_2}+\frac{\pa^2}{\pa x_2^2} \right) \left( \frac{1}{n+1}|u(t,\mx)|^{n}u(t,\mx) \right)=0 \;, \ \ \mx=(x_1,x_2)\,,
	\end{split}
	\end{equation} augmented with the initial data $u_0\in L^1(\R^2)\cap L^\infty(\R^2)$ admits 
	regularisation effects in the sense that for any $q\in (1,q_*)$ and $s\in (0,s_*)$ it holds 
	$$
	L^1\cap L^\infty(\R^2) \ni u_0 \mapsto u\in  W_{loc}^{s,q}(\R^3_+) \,
	$$ 
	where $u$ is the entropy solution to \eqref{TT-1} with $u|_{t=0}=u_0(\mx)$, and $q_*, s_*$ are given 
	in Theorem \ref{t-regularity} for $\alpha=\min\{\frac{1}{2l},\frac{1}{n}\}$ and $d=3$. 
\end{corollary}
\begin{proof}
	We first notice that by introducing the following linear regular change of variables 
	$$
	y_1=x_1+x_2\,, \ \ y_2=x_1-x_2
	$$ 
	we can rewrite \eqref{TT-1} in the form
	\begin{equation*}
%	\label{TT-2}
	\begin{split}
	\pa_t u(t,\my) &+\frac{\pa}{\pa y_1} \Big(\frac{1}{l+1} u^{l+1}(t,\my)\Big)-\frac{\pa^2}{\pa y^2_2} \Big( \frac{1}{n+1}|u(t,\my)|^{n}u(t,\mx) \Big)=0 \,.
	\end{split}
	\end{equation*}
	It is obviously sufficient to show that the solution $u$ admits a regularising effect in these 
	new variables. 
	We will show that this equation satisfies the non-degeneracy condition \eqref{non-deg-stand}, 
	which is sufficient by Theorem \ref{t-regularity} to have a regularising effect 
	(see the end of the proof of Theorem \ref{main-tvr} for the connection between the solution and 
	the average; we also refer to \cite[sections 3 and 4]{TT} where the complete procedure is given). 
	
	In the notation of Definition \ref{def-entropy} (the time variable $t$ is incorporated into the space variables, as it was the case in the mentioned definition) we have
	$$
	f(\lambda)=(1,\lambda^l,0), \ \ a(\lambda)=\operatorname{diag}(0,0,|\lambda|^n)\,.
	$$
	For an arbitrary $M>0$ we will show that the non-degeneracy condition is satisfied on 
	$I:=[-M,M]$ (in fact it is sufficient to consider only $M=\|u_0\|_{L^\infty(\R^2)}$ since 
	the maximum principle ensures that a bounded entropy solution stays in $[-\|u_0\|_{L^\infty(\R^2)},\|u_0\|_{L^\infty(\R^2)}]$).
	For $\delta>0$ and $\mxi=(\xi_0,\xi_1,\xi_2)\in \mathrm{S}^2$, 
	i.e.~$\xi_0^2+\xi_1^2+\xi_2^2=1$, we define
	\begin{align*}
	\Lambda_\delta(\mxi) &:=\Bigl\{\lambda\in I : |\langle f(\lambda)\,|\,\mxi\rangle|^2+\langle a(\lambda)\mxi\,|\,\mxi\rangle \leq\delta\Bigr\} \\
	&\;=\Bigl\{\lambda\in I : (\xi_0+\lambda^l\xi_1)^2+|\lambda|^n\xi_2^2 \leq\delta\Bigr\} \;.
	\end{align*}
	Obviously $\Lambda_\delta(\mxi)=\Lambda_\delta(-\mxi)$. Hence, we can assume that $\xi_1\geq 0$.
	The aim is to prove that there exits $\alpha>0$ such that 
	$\operatorname{meas}\Lambda_\delta(\mxi)\lesssim\delta^\alpha$ for small values of $\delta$.
	
	We shall consider two cases
	$$
	(i) \ \ {\xi_2^2} \geq \frac{1}{4}\,; \qquad (ii) \ \ {\xi_2^2} < \frac{1}{4}\,.
	$$ 
	In the case (i), we clearly have 
	\begin{equation}
	\label{(i)}
	\begin{split}
	\operatorname{meas}\Lambda_\delta(\mxi) &\leq \operatorname{meas}\bigl\{\lambda \in I :\,  
		|\lambda|^n {\xi_2^2} \leq \delta \bigr\}  \\
	& \leq \operatorname{meas}\bigl\{\lambda \in I :\, |\lambda|^n \leq 4\delta \bigr\} 
		\lesssim  \delta^{\frac{1}{n}}\,.
	\end{split} 
	\end{equation} 
	
	In the case (ii), we have
	$$
	{\xi_0^2+\xi_1^2} > 3/4 \,.
	$$ 
	Neglecting the second order term we get
	$$
	\operatorname{meas}\Lambda_\delta(\mxi) \leq
		\operatorname{meas}\bigl\{\lambda \in I :\, |\xi_0+\lambda^l\xi_1| \leq \sqrt{\delta} \bigr\} \,.
	$$
	The term on the right hand side corresponds to the symbol of the one-dimensional scalar conservation law
	studied in \cite[(3.7)]{TT}. Hence, we have 
	\begin{equation}\label{(ii)}
	\operatorname{meas}\Lambda_\delta(\mxi)\lesssim \bigl(\sqrt{\delta}\bigr)^{\frac 1l}= \delta^{\frac{1}{2l}} \,.
	\end{equation}

	Therefore, from \eqref{(i)} and \eqref{(ii)}, we conclude that \eqref{non-deg-stand} holds with ${\alpha}=\min\{\frac{1}{2l},\frac{1}{n}\}$. Using Theorem \ref{t-regularity}, we conclude the proof of the theorem.
\end{proof}

Let us close the paper with a few remarks on non-degeneracy conditions. 

\begin{remark}\label{rem:non-deg}
\item[i)] In the regularity result for general symbols given in \cite{TT}, non-degeneracy 
conditions are composed of two parts: formulas (2.19) and (2.20) there. 
Our condition \eqref{non-deg-stand} can be related to the first one which reads:
there exist $\alpha, \beta>0$ and a finite interval $I\subseteq\R$ 
such that for every $J\gtrsim \N$ and $\delta>0$ it holds
\begin{equation}
\label{c-TT}
\sup\limits_{|\mxi|\sim J} {\rm meas}
\Bigl\{ 
\lambda \in I:\, 
|\langle f(\lambda) \,|\, \mxi \rangle|^2+
\langle a(\lambda)\mxi\,|\, \mxi \rangle^2 \leq\delta^2
\Bigr\}
\lesssim \left(\frac{\delta}{J^\beta}\right)^\alpha. 
\end{equation} 

It is not difficult to see that the condition above implies \eqref{non-deg-stand}. 
Indeed, let us take arbitrary $\delta\in (0,1)$ and $\mxi\in\Sdmj$. 
Let us fix $J\in\N$ such that \eqref{c-TT} is satisfied. 
Since $\delta<1$, we have
\begin{equation*}
\begin{split}
&|\langle f(\lambda) \,|\, \mxi \rangle|^2+
\langle  a(\lambda)\mxi\,|\, \mxi \rangle \leq\delta
\ \ \implies \\
&|\langle f(\lambda) \,|\, \mxi \rangle|^2+
\langle a(\lambda)\mxi\,|\, \mxi \rangle
	\geq |\langle f(\lambda) \,|\, \mxi \rangle|^2+
\langle a(\lambda)\mxi\,|\, \mxi \rangle^2
	\geq \frac{1}{J^2}\Bigl(|\langle f(\lambda) \,|\, \mxi \rangle|^2+
	J^2\langle a(\lambda)\mxi\,|\, \mxi \rangle^2\Bigr) \,.
\end{split}
\end{equation*}
Thus, 
\begin{equation*}
\begin{split}
\operatorname{meas} &
\big\{ 
\lambda \in I:\, 
|\langle f(\lambda) \,|\, \mxi \rangle|^2+
\langle a(\lambda)\mxi\,|\, \mxi \rangle \leq\delta
\big\}\\
& \leq \operatorname{meas}
\big\{ 
\lambda \in I:\, 
|\langle f(\lambda) \,|\, (J\mxi) \rangle|^2+
\langle a(\lambda)(J\mxi)\,|\, (J\mxi) \rangle^2 \leq\delta/J^4
\big\} \lesssim_J \delta^{\alpha/2}\,,
\end{split}
\end{equation*} 
which concludes the proof.

Therefore, with Theorem \ref{t-regularity} we are in position to obtain a regularity result
with much weaker assumptions then in \cite{TT} (e.g.~condition (2.20) is redundant). 
This generalisation turned out to be significant already on the example presented 
in Corollary \ref{TT-ex}. 

However, due to stronger assumptions, when the result of \cite{TT}
(see also \cite{GH}) is applicable, then one can expect better regularity results 
compared to our Theorem \ref{t-regularity}. For instance, in our result the regularity exponent 
$s$ cannot exceed $1/3$, which is not the case with the results of the aforementioned papers. 

\item[ii)] At first glance, one might think, based on Step IV of the proof of Theorem \ref{t-regularity}, 
that the non-degeneracy condition can be further weakened by replacing \eqref{non-deg-stand} with:
there exist $\alpha, R, \delta_0>0$ and $r\in (0,1]$ such that for any $\delta\in (0,\delta_0)$ 
it holds
\begin{equation*}
%\label{nondeg}
\sup\limits_{|\mxi|\geqslant R}  
\operatorname{meas}\Bigl\{\lambda \in I :\, |\langle f(\lambda)\,|\, \mxi'
\rangle |^2+|\mxi|^{1-r} \langle a(\lambda)\mxi' \,|\, \mxi' \rangle  \leq  \delta \Bigr\}\lesssim \delta^{\alpha} \,,
\end{equation*}
where $\mxi'=\mxi/|\mxi|$. However, it is not difficult to see that this condition is in fact 
equivalent with \eqref{non-deg-stand}. 

Nevertheless, if we modify the condition above in the same spirit as the condition 
\eqref{c-TT} is given, i.e.~on the level of \emph{dyadic} blocks, then the result of Theorem \ref{t-regularity}
can be slightly improved. More precisely, the condition should read:
there exist $\alpha, \beta, \delta_0>0$, $r\in (0,1]$ and a finite interval $I\subseteq\R$ 
such that for every $J\gtrsim \N$ and $\delta\in(0,\delta_0)$ it holds
\begin{equation*}
%\label{non-deg-r-J}
\sup\limits_{|\mxi|\sim J} {\rm meas}
\Bigl\{ 
\lambda \in I:\, 
|\langle f(\lambda)\,|\, \mxi'
\rangle |^2+|\mxi|^{1-r} \langle a(\lambda)\mxi' \,|\, \mxi' \rangle \leq\delta^2
\Bigr\}
\lesssim \left(\frac{\delta}{J^\beta}\right)^\alpha. 
\end{equation*} 
Then in \eqref{aroundzero} instead of $\delta^{\alpha/2}$ one would get $\delta^{\alpha/2}
2^{-\frac{1}{2}\alpha\beta K}$.
In this paper, we decided not to go in this direction with the aim of obtaining results with the simplest and most general condition possible. 
\end{remark}

%========================================================================
%===============================================================================
\section{Acknowledgements}\label{sec:thanks}
This work was supported in part by the Croatian Science Foundation under project UIP-2017-05-7249 (MANDphy), and by the project P 35508 of the Austrian Science Fund FWF.

Permanent address of D.~Mitrovi\'c is University of Montenegro, Podgorica, Montenegro.

\end{document}